\documentclass[11pt,reqno]{amsart}
\usepackage{amsmath,amsopn,amssymb,amsthm,multicol}
\usepackage{hyperref}
\usepackage[all]{xy}
\usepackage{graphicx}
\usepackage{graphics}

\DeclareMathOperator{\Ad}{Ad}

\DeclareMathOperator{\Aut}{Aut}

\DeclareMathOperator{\tr}{tr}

\DeclareMathOperator{\Ric}{Ric}

\newcommand{\fr}{\mathfrak}
\newcommand{\al}{\alpha}

\newcommand{\bb}{\mathbb}

\DeclareMathOperator{\SO}{SO}

\DeclareMathOperator{\Sp}{Sp}
\DeclareMathOperator{\SU}{SU}
\DeclareMathOperator{\U}{U}

 \newtheorem{lemma} {Lemma} [section]
\newtheorem{theorem}[lemma]{Theorem} 
 
\newtheorem{prop} [lemma]{Proposition}

\usepackage{array}
\makeatletter
\newcommand{\thickhline}{%
    \noalign {\ifnum 0=`}\fi \hrule height 1pt
    \futurelet \reserved@a \@xhline
}
\newcolumntype{"}{@{\hskip\tabcolsep\vrule width 1pt\hskip\tabcolsep}}
\makeatother

\begin{document}

\title{New homogeneous Einstein metrics on quaternionic Stiefel manifolds} 
\author{Andreas Arvanitoyeorgos, Yusuke Sakane  and Marina Statha}
\address{University of Patras, Department of Mathematics, GR-26500 Rion, Greece}
\email{arvanito@math.upatras.gr}
 \address{Osaka University, Department of Pure and Applied Mathematics, Graduate School of Information Science and Technology, Toyonaka, 
Osaka 560-0043, Japan}
 \email{sakane@math.sci.osaka-u.ac.jp}
\address{University of Patras, Department of Mathematics, GR-26500 Rion, Greece}
\email{statha@master.math.upatras.gr} 
\medskip

\begin{abstract}
  
We consider invariant Einstein metrics on the  quaternionic Stiefel manifolds $V_p\bb{H} ^n$  of all orthonormal $p$-frames in $\bb{H}^n$.  This manifold is diffeomorphic to the homogeneous space $\Sp(n)/\Sp(n-p)$ and its
isotropy representation contains equivalent summands.  
We obtain new Einstein metrics on $V_{p}\bb{H}^{n} \cong \Sp(n)/\Sp(n-p)$, where $n = k_1 + k_2 + k_3$ and $p = n-k_3$.  We view $V_{p}\bb{H}^{n}$ as a total space over the genaralized Wallach space $\Sp(n)/(\Sp(k_1)\times\Sp(k_2)\times\Sp(k_3))$ and  over the generalized flag manifold $\Sp(n)/(\U(p)\times\Sp(n-p))$.


   \end{abstract}

\maketitle


\section{Introduction}
\markboth{Andreas Arvanitoyeorgos, Yusuke Sakane and Marina Statha}{New homogeneous Einstein metrics on quaternionic Stiefel manifolds}

A Riemannian manifold $(M, g)$ is called Einstein if the metric $g$ satisfies the condition $\Ric (g) = \lambda\cdot g$ for some $\lambda\in\bb{R}$.  We refer to \cite{Be} and \cite{W1}, \cite{W2} for old and new results on homogeneous Einstein manifolds.  The structure of the set of invariant Einstein metrics on a given homogeneous space is still not very well understood in general.  The situation is only clear for few classes of homogeneous spaces, such as isotropy irreducible homogeneous spaces, low dimensional examples, certain flag manifolds.  For an arbitrary compact homogeneous space $G/H$ it is not clear if the set of invariant Einstein metrics (up to isometry and up to scaling) is finite or not.  A finiteness conjecture states that this set is in fact finite if the isotropy representation of $G/H$ consists of pairwise inequivalent irreducible subrepresentations (\cite{BWZ}).  In the case where the isotropy representation contains some equivalent subrepresentations (isotropy summands) then the diagonal metrics are not unique.  In \cite{ADN1} the authors introduced a method for proving existence of homogeneous Einstein metrics by assuming additional symmetries.  In \cite{St} a systematic and organized description of such metrics is presented.

In the present article we are interested to invariant Einstein metrics on homogeneous spaces $G/H$ whose isotropy representation is decomposed into a sum of irreducible but possibly equivalent summands.  Typical examples of such
homogeneous spaces are the Stiefel manifolds.

For new and old results about Einstein metrics on real Stiefel manifolds $V_{p}\bb{R}^{n}=\SO(n)/\SO(n-p)$ we 
refer to \cite{ArSaSt1}.
The first  invariant Einstein metrics on the quaternionic Stiefel manifolds $V_{p}\bb{H}^{n}=\Sp(n)/\Sp(n-p)$ were obtained by G. Jensen in \cite{J2}, by using Riemannian submersions.
Invariant Einstein metrics on the two marginal cases $V_{1}\bb{H}^{n}=\mathbb{S}^{4n-1}$ and
$V_{n}\bb{H}^{n}=\Sp(n)$ have been studied in \cite{Zi} and \cite{ArSaSt2} respectively. 

In \cite{ADN2} the first author, V.V. Dzhepko and Yu. G. Nikonorov proved that for $s>1$ and $\ell \ge k\ge 1$ the Stiefel manifold $\Sp(s k+\ell)/\Sp(\ell)$ admits at least four $\Sp(s k+\ell)\times (\Sp(k))^s$-invariant Einstein metrics, two of which are Jensen's metrics.  We call the two Einstein metrics,  different from Jensen's metrics, as {\itshape ADN metrics}.

In the present paper we obtain new invariant Einstein metrics on $V_{p}\bb{H}^{n} \cong \Sp(n)/\Sp(n-p)$ where $n = k_1 + k_2 + k_3$ and $p = n-k_3$.  
We view $V_{p}\bb{H}^{n}$, firstly as a total space over the genaralized Wallach space $\Sp(n)/(\Sp(k_1)\times\Sp(k_2)\times\Sp(k_3))$ and secondly as total space over the generalized flag manifold $\Sp(n)/(\U(p)\times\Sp(n-p))$, and
 use invariant metrics described by corresponding inner products (\ref{metric1}) and (\ref{metric2}) (cf. Section 3).
Then we use the method in \cite{ADN1} or \cite{St}, appropriately adjusted.

The main results related to the first fibration are the following:

\smallskip 
\noindent
{\bf Theorem A.}
For $n=3, 4$ the Stiefel manifold $V_{2}\bb{H}^{n}$ admits:
\begin{itemize}
\item[(1)] Eight invariant Einstein metrics which are determined by $\Ad(\Sp(1)\times\Sp(1)\times\Sp(1))$-invariant inner products of the form (\ref{metric1}).  Four of them are new, two are Jensen's metrics and the other two are ADN metrics.
\item[(2)] Eight invariant Einstein metrics which are determined by $\Ad(\Sp(1)\times\Sp(1)\times\Sp(2))$-invariant inner products of the form (\ref{metric1}).  Four of them are new, two are Jensen's metrics and the other two are ADN metrics.
\end{itemize}

\smallskip  
 The main result related to the second fibration is the following:

\noindent
{\bf Theorem B.}
{\it For $ \displaystyle 2  \leq p \leq \frac{3}{4} n$, there exist two  invariant Einstein metrics on $ \Sp(n)/\Sp(n-p)$  of the form {\em (\ref{metric2})} which are different from Jensen's metrics.} 

\smallskip
The proofs of the above theorems are given in Section 4 (Theorems 4.1 and 4.2).
We also make a conjecture about Einstein metrics on the Stiefel manifolds $V_{n-1}\bb{H}^{n}$ ($n\ge 3$) and $V_{n-2}\bb{H}^{n}$ ($n\ge 5$), (cf. Table 1).

\medskip

\noindent
{\bf Acknowledgements.} 
The work was supported by Grant $\# E.037$ from the Research Committee of the University of Patras (Programme K. Karatheodori) and JSPS KAKENHI Grant Number JP16K05130.

\section{A special class of $G$-invariant metrics on $G/H$ and the Ricci tensor}

Let $G$ be a compact Lie group and $H$ a closed subgroup so that $G$ acts transitively on the homogeneous space $G/H$.  Then the homogeneous space $G/H$ is reductive, because we can take $\fr{m} = \fr{g}^{\perp}$ where $\Ad(H)\fr{m}\subset \fr{m}$ with respect to an $\Ad$-invariant scalar product on $\fr{g}$.  So the Lie algebra $\fr{g}$ can be written as $\fr{g} = \fr{h}\oplus\fr{m}$.  The tangent space of $G/H$ at the ${\it o} = eH\in G/H$ is canonically identified with $\fr{m}$.  For $G$ semisimple, the negative of the Killing form $B$ of $\fr{g}$ is an $\Ad(G)$-invariant scalar product, therefore we can choose the above decomposition with respect to this form.  A Riemannian metric $g$ on $G/H$ is called $G$-invariant if the diffeomorphism $\tau_{\al} : G/H \to G/H,$ $\tau_{\al}(g H) = \al g H$ is an isometry.  Any such metric is to one-to-one correspondence with an $\Ad(H)$-invariant scalar product $\langle\cdot,\cdot\rangle$ on $\fr{m}$ and fixed points $(\mathcal{M}^{G})^{\Phi_{H}}$ of the action $\Phi_{H}$ $=\{\Ad(h)|_{\fr{m}} : h\in H\}$ $\subset \Phi = \{\Ad(n)|_{\fr{m}} :  n \in N_{G}(H)\}$ $\subset$ $\Aut(\fr{m})$ on $\mathcal{M}^{G}$.  In the special case where $H = \{e\}$ then $N_{G}(H) = G$, thus the fixed points $(\mathcal{M}^{G})^{\Phi_{H}}$ are the $\Ad(G)$-invariant inner products on $\fr{g}$.  These correspond to the bi-invariant metrics on the Lie group $G$. 

The isotropy representation $\chi : H \to \Aut(\fr{m})$ of the reductive homogeneous space $G/H$ coincides with the restriction of the adjoint representation of $H$ on $\fr{m}$.  We assume that $\chi$ decompose into a direct sum of irreducible subrepresentation $\chi \cong \chi_{1}\oplus\cdots\chi_{s}$, so the tangent space splits into a direct sum of $\Ad(H)$-invariant subspaces 
\begin{equation}\label{diaspasi}
\fr{m} = \fr{m}\oplus\cdots\oplus\fr{m}_{s}.
\end{equation}
In this case the $G$-invariant metrics are determined by the diagonal $\Ad(H)$-invariant scalar products of the form
\begin{equation}\label{diagonia}
\langle\cdot,\cdot\rangle = x_{1}(-B)|_{\fr{m}_{1}} + \cdots + x_{s}(-B)|_{\fr{m}_{s}}, x_{i}\in\bb{R}^{+}.
\end{equation}
If some of the subrepresentations $\chi_{i}$ are equivalent then decomposition (\ref{diaspasi}) is not unique.  Hence the $\Ad(H)$-invariant scalar product is not necessary diagonal.  For this case we can choose a closed subgroup $K$ of $G$ such that $H\subset K\subset N_{G}(H)$ and search for $\Ad(K)$-invariant scalar products on $\fr{m}$ which correspond to a subset $\mathcal{M}^{G, K}$ of $G$-invariant metrics on $G/H$, sometimes also called $\Ad(K)$-invariant metrics.  The benefit of such metrics is that they are diagonal metrics on the homogeneous space.  The next proposition gives a possible way to choose such a subgroup $K$ of $G$.

\begin{prop}\label{subset}
Let $K$  be a subgroup of $G$ with $H\subset K \subset G$ and such that $K = L\times H$, for some subgroup $L$ of $G$.  Then $K$ is contained in $N_{G}(H)$.
\end{prop}
 
Now we describe the Ricci tensor for the diagonal metrics of the form (\ref{diagonia}).  Every $G$-invariant symmetric covariant 2-tensor on $G/H$ are of the same form as the Riemannian metrics (although they are not necessarly positive definite).  In particular, the Ricci tensor $r$ of a $G$-invariant Riemannian metric on $G/H$ is of the same form (\ref{diagonia}), that is
$$
r = r_1 x_1(-B)|_{\fr{m}_1} + \cdots r_{s} x_{s}(-B)|_{\fr{m}_s}.
$$
Let $\lbrace e_{\alpha} \rbrace$ be a $(-B)$-orthonormal basis adapted to the decomposition of $\frak m$, i.e. $e_{\alpha} \in {\frak m}_i$ for some $i$, and $\alpha < \beta$ if $i<j$.  We put ${A^\gamma_{\alpha \beta}}= -B\left(\left[e_{\alpha},e_{\beta}\right],e_{\gamma}\right)$ so that $\left[e_{\alpha},e_{\beta}\right]
= \displaystyle{\sum_{\gamma} A^\gamma_{\alpha \beta} e_{\gamma}}$ and set $A_{ijk}:=\displaystyle{k \brack {ij}}=\sum (A^\gamma_{\alpha \beta})^2$, where the sum is taken over all indices $\alpha, \beta, \gamma$ with $e_\alpha \in {\frak m}_i,\ e_\beta \in {\frak m}_j,\ e_\gamma \in {\frak m}_k$ (cf. \cite{WZ}).  Then the positive numbers $A_{ijk}$ are independent of the 
$(-B)$-orthonormal bases chosen for ${\frak m}_i, {\frak m}_j, {\frak m}_k$, and 
$A_{ijk}\ =\ A_{jik}\ =\ A_{kij}. $

Let $ d_k= \dim{\frak m}_{k}$. Then we have the following:

\begin{lemma}\label{ric2}\textnormal{(\cite{PS})}
The components ${ r}_{1}, \dots, {r}_{s}$ of the Ricci tensor ${r}$ of the metric $ \langle\cdot, \cdot \rangle$ of the
form {\em (\ref{diagonia})} on $G/H$ are given by 
\begin{equation}
{r}_k = \frac{1}{2x_k}+\frac{1}{4d_k}\sum_{j,i}
\frac{x_k}{x_j x_i} A_{jik}
-\frac{1}{2d_k}\sum_{j,i}\frac{x_j}{x_k x_i} A_{kij}
 \quad (k= 1,\ \dots, s),    \label{eq51}
\end{equation}
where the sum is taken over $i, j =1,\dots, s$.
\end{lemma} 
Since by assumption the submodules $\fr{m}_{i}, \fr{m}_{j}$ in the decomposition (\ref{diaspasi}) are matually non equivalent for any $i\neq j$, it is $r(\fr{m}_{i}, \fr{m}_{j})=0$ whenever $i\neq j$.  Thus by Lemma \ref{ric2} it follows that $G$-invariant Einstein metrics on $M=G/H$ are exactly the positive real solutions $g=(x_1, \ldots, x_s)\in\bb{R}^{s}_{+}$  of the  polynomial system $\{r_1=\lambda, \ r_2=\lambda,  \ldots,  r_{s}=\lambda\}$, where $\lambda\in \bb{R}_{+}$ is the Einstein constant.

\section{The Stiefel manifold $V_{p}\bb{H}^{n}\cong \Sp(n)/\Sp(n-p)$}

We embed the Lie algebra 
\begin{equation*}
\mathfrak{sp}(n-p)=\left\{\begin{pmatrix}
X & -{}\bar{Y}\\
Y & \bar{X}
\end{pmatrix}\ \Big\vert 
 \begin{array}{l}X\in\mathfrak{u}(n-p), \\ 
Y\   (n-k)\times (n-p)\ \mbox{complex symmetric matrix}
\end{array}
\right\}
\end{equation*} 
of the Lie group $\Sp(n-p)$ in the Lie algebra $\fr{sp}(n)$ of $\Sp(n)$ as
\begin{equation*}
\fr{sp}(n-p)\ni \begin{pmatrix}
X & -{}\bar{Y}\\
Y & \bar{X}
\end{pmatrix} \hookrightarrow  \begin{pmatrix}
 0& 0 &\vline & 0& 0\\
 0 & X & \vline & 0 & -\bar{Y}\\
\hline
0 & 0 &\vline & 0 & 0  \\
0 & Y & \vline & 0 &  \bar{X}
 \end{pmatrix}  \in \fr{sp}(n).
\end{equation*}
The Killing form of $\fr{sp}(n)$ is $B(X, Y)=2(n+1)\tr XY$.  Then with respect to $-B$ we can find the $\Ad(\Sp(n-p))$-invariant subspace $\fr{m}\cong T_{o}(\Sp(n)/\Sp(n-p))$ such that $\fr{sp}(n) = \fr{sp}(n-p)\oplus\fr{m}$.

Next we review the isotropy representation of $G/H = \Sp(n)/\Sp(n-p)\cong$ $V_{p}\bb{H}^{n}$.  Let $\nu_{n} : \Sp(n) \to \Aut(\bb{C}^{2n})$ be the standard representation of  $\Sp(n)$ and $\Ad^{\Sp(n)}\otimes\bb{C} = S^{2}\nu_{n}$ the complexified adjoint representation of $\Sp(n)$, where $S^{2}$ is the second symmetric power.  For the isotropy representation $\chi : \Sp(n) \to \Aut(\fr{m})$ of the homogeneous space $\Sp(n)/\Sp(n-p)$
we have:
\begin{eqnarray*}
\Ad^{Sp(n)}\otimes\bb{C}\big|_{\Sp(n-p)} &=& S^{2}\nu_{n}\big|_{\Sp(n-p)} = S^{2}(\nu_{n-p}\oplus p\oplus p)\\
                                     &=& S^{2}\nu_{n-p} \oplus S^{2}(p\oplus p) \oplus \big(\nu_{n-p}\otimes(p \oplus p)\big)\\
                                     &=& S^{2}\nu_{n-p} \oplus S^{2}(p\oplus p) \oplus (\nu_{n-p}\otimes p)\oplus(\nu_{n-p}\otimes p)\\
                                     &=& S^{2}\nu_{n-p} \oplus S^{2}(p\oplus p) \oplus \underbrace{\nu_{n-p}\oplus\cdots\oplus \nu_{n-p}}_{p-\mbox{times}} \oplus\underbrace{\nu_{n-p}\oplus\cdots\oplus \nu_{n-p}}_{p-\mbox{times}}\\
                                     &=& S^{2}\nu_{n-p}\oplus\underbrace{1\oplus\cdots\oplus 1}_{{{2+2p-1}\choose{2}}-\mbox{times}}\oplus\underbrace{\nu_{n-p}\oplus\cdots\oplus \nu_{n-p}}_{2p-\mbox{times}}.
\end{eqnarray*}
In the first line above, $p$ denotes the direct sum $1\oplus \cdots\oplus1$ of 1-dimensional trivial representations. Hence
\begin{equation}\label{isotropiki}
\chi\otimes \bb{C} = 1\oplus\cdots\oplus 1\oplus\nu_{n-p}\oplus\cdots\oplus\nu_{n-p}. 
\end{equation}
This decomposition induces an $\Ad(\Sp(n-p))$-invariant decomposition of $\fr{m}\otimes\bb{C}$ as
\begin{equation}\label{diaspefapt}
\fr{m}\otimes\bb{C} = \fr{m}_{1}\oplus\fr{m}_{2}\oplus\cdots\oplus\fr{m}_{s},
\end{equation}
where the first $\displaystyle{{{2+2k-1}\choose{2}}}$ $\Ad(\Sp(n-p))$-submodules are 1-dimensional and the rest  $2p$ are $(n-p)$-dimensional.  Note that the decomposition (\ref{isotropiki}) contains equivalent subrepresentations so a complete description of all $\Sp(n)$-invariant metrics on $\Sp(n)/\Sp(n-p)$ is not easy. 

\subsection{$V_{p}\bb{H}^{n}$ as total space over a generalized Wallach space} 

Let $n = k_1 + k_2 + k_3$ and $p = k_1 + k_2$. We consider the closed subgroup $K = \Sp(k_1)\times\Sp(k_2)\times\Sp(k_3)$ of $\Sp(n)$.  Then from Proposition \ref{subset} $K$ is contained in $N_{\Sp(n)}(\Sp(n - p))$.  We consider the fibration 
$$
\Sp(k_1)\times\Sp(k_2) \to G/H=\Sp(n)/\Sp(n-p) \to \Sp(n)/(\Sp(k_1)\times\Sp(k_2)\times\Sp(k_3))
$$

Let $\frak{a}$ and $\frak{p}$ be the orthogonal complements of $\fr{sp}(k_3)$ in $\fr{sp}(k_1)\oplus\fr{sp}(k_2)\oplus\fr{sp}(k_3)$, and of $\fr{sp}(k_1)\oplus\fr{sp}(k_2)\oplus\fr{sp}(k_3)$ in $\fr{sp}(n)$, with respect to the negative of the Killing form of $\fr{sp}(n)$.  The spaces $\fr{a}$ and $\fr{p}$ are called \textit{vertical} and \textit{horizontal} subspaces of $\fr{g}$.  Hence the tangent space of quaternionic Stiefel manifold $G/H$ can be written as $\fr{m} = \fr{a}\oplus\fr{p}$.  We observe that $\fr{p}$ is the tangent space of the generalized Wallach space $G/K = \Sp(n)/(\Sp(k_1)\times\Sp(k_2)\times\Sp(k_3))$.  Actually for $i=1, 2, 3$, we embed the Lie subalgebras 
\begin{equation*}
\mathfrak{sp}(k_i)=\left\{\begin{pmatrix}
X_i & -{}\bar{Y}_i\\
Y_i & \bar{X}_i
\end{pmatrix}\ \Big\vert 
 \begin{array}{l}X_i\in\mathfrak{u}(k_i), \\ 
Y_i\ \mbox{ is \ a}\  k_i\times k_i\ \mbox{complex symmetric matrix}
\end{array}
\right\}\ 
\end{equation*} 
in the Lie algebra $\mathfrak{sp}(k_1+ k_2+k_3)$ as follows: 
{\small\small
\begin{equation*}
 \left\{\begin{pmatrix}
 X_1 & 0 & 0 &\vline & -\bar{Y}_{1} & 0 & 0\\
 0 & 0 & 0 & \vline & 0 & 0 & 0\\
0 & 0 & 0 & \vline & 0 & 0 & 0\\
\hline
Y_1 & 0 & 0 &\vline & \bar{X}_1 & 0 & 0\\
0 & 0 & 0 & \vline & 0 & 0 & 0\\
0 & 0 & 0 & \vline & 0 & 0 & 0
 \end{pmatrix} \right\}, 
  \left\{\begin{pmatrix}
0 & 0 & 0 &\vline &  0 & 0 & 0\\
 0 & X_2 & 0 & \vline & 0 &-\bar{Y}_{2}  & 0\\
0 & 0 & 0 & \vline & 0 & 0 & 0\\
\hline
0  & 0 & 0 &\vline & 0 & 0 & 0\\
0 & Y_2 & 0 & \vline & 0 & \bar{X}_2  & 0\\
0 & 0 & 0 & \vline & 0 & 0 & 0
 \end{pmatrix} \right\}, 
 \left\{\begin{pmatrix}
 0 & 0 & 0 &\vline & 0  & 0 & 0\\
 0 & 0 & 0 & \vline & 0 & 0 & 0\\
0 & 0 & X_3 & \vline & 0 & 0 & -\bar{Y}_{3}\\
\hline
0 & 0 & 0 &\vline & 0  & 0 & 0\\
0 & 0 & 0 & \vline & 0 & 0 & 0\\
0 & 0 & Y_3 & \vline & 0 & 0 & \bar{X}_3
 \end{pmatrix} \right\}.    
 \end{equation*} 
 }
Then the tangent space $\mathfrak{p}$ of $G/K$ is given by  $\mathfrak{k}^{\perp} $ in $\mathfrak{g} = \mathfrak{sp}(k_1+ k_2+k_3)$ with respect to the $\Ad(G)$-invariant inner product $-B$.  If we denote by $M(a,b)$ the set of all $a \times b$ matrices, then we see that $\mathfrak{p}$ is given by  
\begin{equation*}
\mathfrak{p}=  {\small \left\{\begin{pmatrix}
 0 & {A}_{12} & {A}_{13} & 0 & -\bar{B}_{12} & -\bar{B}_{13}\\
 -{}^{t}_{}\!\bar{A}_{12} & 0 & {A}_{23} & -{}^{t}_{}\!\bar{B}_{12} & 0 & -\bar{B}_{23}\\
 -{}^{t}_{}\!\bar{A}_{13} & -{}^{t}_{}\!\bar{A}_{23} & 
    0 & -{}^{t}_{}\!\bar{B}_{13} & -{}^{t}_{}\!\bar{B}_{23} & 0\\
0 & {B}_{12} & {B}_{13} & 0 & \bar{A}_{12} & \bar{A}_{13}\\
 {}^{t}_{}\!{B}_{12} & 0 & {B}_{23} & -{}^{t}_{}\!{A}_{12} & 0 & \bar{A}_{23}\\
 {}^{t}_{}\!{B}_{13} & {}^{t}_{}\!{B}_{23} & 
    0 & -{}^{t}_{}\!{A}_{13} & -{}^{t}_{}\!{A}_{23} & 0
 \end{pmatrix} \  \Bigg\vert \  \begin{array}{l} {A}_{ij}, {B}_{ij} \in M(k_i, k_j)  \\  ( 1 \leq i < j \leq 3  )
  \end{array} \right\}. } 
 \end{equation*}

If $k_1, k_2, k_3$ are distinct then the isotropy representation of $G/K$ can be written as a direct sum of three non equivalent subrepresentations.  More precisely, let $p_{i} = \nu_{k_{i}}\circ\sigma_{k_{i}}$ be the standard representation of $K$ i.e.  
$$ \Sp(k_{1})\times\Sp(k_{2})\times\Sp(k_{3})\stackrel{\sigma_{k_{i}}}{\longrightarrow}\Sp(k_{i})\stackrel{\nu_{k_{i}}}{\longrightarrow}\Aut(\bb{C}^{2k_{i}}).
$$
By using the relation $\Ad^{G}\otimes\bb{C}|_{K} = (\Ad^{K}\otimes\bb{C})\oplus(\chi\otimes\bb{C})$ we have that
\begin{eqnarray*}
\Ad^{G}\otimes\bb{C}\big|_{K} &=& S^{2}\nu_{k_{1}+k_{2}+k_{3}}\big|_{K} =  S^{2}(p_{k_{1}} \oplus p_{k_{2}} \oplus p_{k_{3}})\\
&=&S^{2}p_{k_{1}} \oplus S^{2}p_{k_{2}} \oplus S^{2}p_{k_{3}} \oplus (p_{k_{2}}\otimes p_{k_{3}})\oplus (p_{k_{1}}\otimes p_{k_{2}}) \oplus (p_{k_{1}}\otimes p_{k_{3}})\\
&=& (\Ad^{K}\otimes\bb{C})\oplus(p_{k_{1}}\otimes p_{k_{2}})\oplus (p_{k_{1}}\otimes p_{k_{3}})\oplus (p_{k_{2}}\otimes p_{k_{3}}).
\end{eqnarray*}
Hence 
$
\chi\otimes\bb{C} = (p_{k_{1}}\otimes p_{k_{2}})\oplus (p_{k_{1}}\otimes p_{k_{3}})\oplus (p_{k_{2}}\otimes p_{k_{3}})
$
where the dimensions of the subrepresentations are respectively $2k_{1}k_{2},$ $2k_{1}k_{3}$ and $2k_{2}k_{3}$.  Therefore, the complexified tangent space of  $G/K$ is expressed as a direct sum of three non equivalent irreducible submodules as  $\fr{n}\otimes\bb{C} = \fr{n}_{1} \oplus \fr{n}_{2}\oplus \fr{n}_{3}$.  Hence, the real tangent space can be written as
$\fr{p} = \fr{p}_{12} \oplus\fr{p}_{13}\oplus\fr{p}_{23},$ where $\fr{p}_{12}\otimes\bb{C} = \fr{n}_{1}, \fr{p}_{13}\otimes\bb{C} = \fr{n}_{2}, \fr{p}_{23}\otimes\bb{C} = \fr{n}_{3}$ 
and the dimensions of $\fr{p}_{ij}, i\neq j$ are $\dim(\fr{p}_{ij}) = 4k_{i}k_{j}$.  We set
\begin{equation*}
 \mathfrak{p}_{12}= \left\{\begin{pmatrix}
 0 & {A}_{12} & 0 & 0 & -\bar{B}_{12} & 0\\
 -{}^{t}_{}\!\bar{A}_{12} & 0 & 0 & -{}^{t}_{}\!\bar{B}_{12} & 0 & 0\\
 0 & 0 & 0 & 0 & 0 & 0\\
0 & {B}_{12} & 0 & 0 & \bar{A}_{12} & 0\\
 {}^{t}_{}\!{B}_{12} & 0 & 0 & -{}^{t}_{}\!{A}_{12} & 0 & 0\\
 0 & 0 & 
    0 & 0 & 0 & 0
 \end{pmatrix}  \Bigg\vert \   {A}_{12}, {B}_{12} \in M(k_1, k_2) 
 \right\},
  \end{equation*} 
 \begin{equation*}  
 \mathfrak{p}_{13}= \left\{\begin{pmatrix}
 0 & 0 & {A}_{13} & 0 & 0 & -\bar{B}_{13}\\
 0 & 0 & 0 & 0 & 0 & 0\\
 -{}^{t}_{}\!\bar{A}_{13} & 0 & 
    0 & -{}^{t}_{}\!\bar{B}_{13} & 0 & 0\\
0 & 0 & {B}_{13} & 0 & 0 & \bar{A}_{13}\\
 0 & 0 & 0 & 0 & 0 & 0\\
 {}^{t}_{}\!{B}_{13} & 0 & 
    0 & -{}^{t}_{}\!{A}_{13} & 0 & 0
 \end{pmatrix} \Bigg\vert \   {A}_{13}, {B}_{13} \in M(k_1, k_3) 
 \right\},
 \end{equation*} 
 \begin{equation*}
 \mathfrak{p}_{23}= \left\{\begin{pmatrix}
 0 & 0 & 0 & 0 & 0 & 0\\
 0 & 0 & {A}_{23} & 0 & 0 & -\bar{B}_{23}\\
 0 & -{}^{t}_{}\!\bar{A}_{23} & 
    0 & 0 & -{}^{t}_{}\!\bar{B}_{23} & 0\\
0 & 0 & 0 & 0 & 0 & 0\\
 0 & 0 & {B}_{23} & 0 & 0 & \bar{A}_{23}\\
 0 & {}^{t}_{}\!{B}_{23} & 
    0 & 0 & -{}^{t}_{}\!{A}_{23} & 0
 \end{pmatrix} \Bigg\vert \   {A}_{23}, {B}_{23} \in M(k_2, k_3) 
 \right\}.
 \end{equation*}

Hence for the tangent space $\fr{m}$ of $G/H$ we have the decomposition 
\begin{equation}\label{lie}
\fr{m} = \fr{a}\oplus\fr{p} = \fr{sp}(k_1)\oplus\fr{sp}(k_2)\oplus\fr{p}_{12}\oplus\fr{p}_{13}\oplus\fr{p}_{23}.
\end{equation}
So for $k_i$ distinct the $G$-invariant metrics $\check{g}$ on $G/K$ are determined by the following $\Ad(K)$-invariant scalar products on $\fr{p}$:
$$
(\cdot,\cdot) = x_{12}(-B)|_{\fr{p}_{12}} + x_{13}(-B)|_{\fr{p}_{13}} + x_{23}(-B)|_{\fr{p}_{23}}.
$$
Also, any $\Ad(H)$-invariant inner product on $\fr{a}$ defines a $K$-invariant metric $\hat{g}$ on $K/H$.  The direct sum of these inner products on $\fr{m} = \fr{a}\oplus\fr{p}$ defines a $G$-invariant metric
$g =\hat{g}+\check{g}$
on the Stiefel manifold $G/H$, called \textit{submersion metric}.  This metric is determined by the following $\Ad(K)$-invariant inner product on $\fr{m}$:
\begin{equation} \label{metric1} 
\begin{array}{lll}
\langle\cdot, \cdot\rangle &=& x_1 \, (-B)|_{\fr{sp}(k_1)}+ x_2 \, (-B)|_{ \fr{sp}(k_2)} 
+ x_{12} \, (-B)|_{ \fr{p}_{12}}+ x_{13} \,  (-B)|_{ \fr{p}_{13}} + x_{23} \,  (-B)|_{ \fr{p}_{23}},
\end{array}
\end{equation}
where $x_1, x_2, x_{ij}$, $i,j = 1,2,3$ with $i\neq j$, belong to $\bb{R}^{+}$.  In general the submersion metric corresponds to an element of $(\mathcal{M}^G)^{\Phi_K}$, as defined in Section 2.  

We set in the decomposition (\ref{lie}) $\fr{sp}(k_1) = \fr{p}_{1}$ and $\fr{sp}(k_2) = \fr{p}_{2}$.  Then we see that the following relations hold:
\begin{lemma}\label{brackets}  
The submodules in the decomposition {\em (\ref{lie})} satisfy the following bracket relations$:$ 
\begin{center}
\begin{tabular}{lll}
$[ \fr{p}_1, \fr{p}_1] = \fr{p}_1,$  &   $[ \fr{p}_2, \fr{p}_2] = \fr{p}_2,$  & $[ \fr{p}_1, \fr{p}_{12}] = \fr{p}_{12},$ \\
 $[ \fr{p}_1, \fr{p}_{13}] = \fr{p}_{13},$  & $[\fr{p}_2, \fr{p}_{12}] =  \fr{p}_{12},$    & $[ \fr{p}_2, \fr{p}_{23}] =\fr{p}_{23},$
\\ 
$ [ \fr{p}_{12}, \fr{p}_{23}] \subset \fr{p}_{13},$   &   $[ \fr{p}_{13}, \fr{p}_{23}] \subset \fr{p}_{12},$  &
 $[\fr{p}_{12}, \fr{p}_{13}] \subset \fr{p}_{23},$\\ 
$[ \fr{p}_{12}, \fr{p}_{12}] \subset \fr{p}_{1}\oplus  \fr{p}_{2},$   &   $[ \fr{p}_{13}, \fr{p}_{13}] \subset \fr{p}_{1} \oplus \fr{p}_{3},$  &
 $[\fr{p}_{23}, \fr{p}_{23}] \subset \fr{p}_{2} \oplus \fr{p}_{3}$, 
\end{tabular}
\end{center}  
and the other bracket relations are zero.
\end{lemma}

\subsection{$V_{p}\bb{H}^{n}$ as total space over a generalized flag manifold}

Let $n = k_1+k_2+k_3$ and $p = k_1+k_2$.  Now we consider the closed subgroup $K = \U(p)\times\Sp(n-p)$ of $\Sp(n)$.  From Proposition \ref{subset} we have that $K\subset N_{\Sp(n)}(\Sp(n-p))$.  We consider the fibration
$$
\U(p) \to G/H=\Sp(n)/\Sp(n-p) \to \Sp(n)/(\U(p)\times\Sp(n-p)).
$$
The fiber $U(p)$ is diffeomorphic to the Lie group $\U(1)\times\SU(p)$ so the vertical subspace $\fr{h}$ of $\fr{sp}(n)$ is written as direct sum $\fr{h} = \fr{h}_0\oplus\fr{h}_1$, where $\fr{h}_0$ is the center of $\U(p)$.  We set $d_0 = \dim(\fr{h}_0) = 1$ and $d_1 = \dim(\fr{h}_1) = p^{2} - 1$.  We also observe that the horizontal subspace $\fr{p}$ of $\fr{sp}(n)$ is the tangent space of generalized flag manifold $\Sp(n)/(\U(p)\times\Sp(n-p))$.  So the isotropy representation of this space given as follows:
\begin{eqnarray*}
\Ad^{\Sp(n)}\otimes\bb{C}\big|_{\U(p)\times\Sp(n-p)} &=& S^{2}\nu_{n}\big|_{\U(p)\times\Sp(n-p)} = S^{2}(\mu_{p}\oplus\bar{\mu}_{p}\oplus\nu_{n-p})\\
&=& S^{2}\nu_{n-p}\oplus (\mu_{p}\otimes\bar{\mu}_{p}) \oplus S^{2}\mu_{p}\oplus S^{2}\bar{\mu}_{p} \oplus (\mu_{p}\otimes\nu_{n-p})\oplus(\bar{\mu}_{p}\otimes\nu_{n-p})\\
&=& \Ad^{\U(p)\times\Sp(n-p)}\otimes\bb{C} \oplus S^{2}\mu_{p}\oplus S^{2}\bar{\mu}_{p} \oplus (\mu_{p}\otimes\nu_{n-p})\oplus(\bar{\mu}_{p}\otimes\nu_{n-p}).
\end{eqnarray*}
In the above calculations $\mu_{p} : \U(p)\to \Aut(\bb{C}^{p})$ is the standard representation of Lie group $\U(p)$ and $\Ad^{\U(p)}\otimes\bb{C} = \mu_{p}\otimes\bar{\mu}_{p}$ the complexified adjoint representation of $\U(p)$.  Therefore the complexified isotropy representation of the generalized flag manifold $G/K$ is $\chi\otimes\bb{C} = S^{2}\mu_{p}\oplus S^{2}\bar{\mu}_{p} \oplus (\mu_{p}\otimes\nu_{n-p})\oplus(\bar{\mu}_{p}\otimes\nu_{n-p})$.  The dimension of the first 
two subrepresentations is $2p(n-p)$ and of the rest two is $\binom{p+1}{2}$.  Also the representations $\mu_{p}\otimes\nu_{n-p}$ and $\bar{\mu}_{p}\otimes\nu_{n-p}$ are conjugate to each other and the same holds for the representations $S^{2}\mu_{p}$ and $S^{2}\bar{\mu}_{p}$.  Thus $\fr{p}$ decomposes in two real $\Ad(K)$-invariant irreducible submodules $\fr{p}_1$ and $\fr{p}_2$
of dimension $d_2 = \dim(\fr{p}_1) = 4p(n-p)$ and $d_3 = \dim(\fr{p}_3) = p(p+1)$.  So the tangent space $\fr{m}$ of the  Stiefel manifold $\Sp(n)/\Sp(n-p)$ can be expressed as
\begin{eqnarray}\label{diaspasi1}
\fr{m} = \fr{h}\oplus\fr{p} = \fr{h}_0\oplus\fr{h}_1\oplus\fr{p}_1\oplus\fr{p}_2 = \fr{n}_0\oplus\fr{n}_1\oplus\fr{n}_2\oplus\fr{n}_3.
\end{eqnarray}
In this case the submersion metric on the Stiefell manifold $G/H$ is determined by the following $\Ad(K)$-invarinat inner product on $\fr{m}$:
\begin{equation}\label{metric2}
\langle\cdot,\cdot\rangle = u_{0}(-B)|_{\fr{n}_0} + u_1(-B)|_{\fr{n}_1} + u_2(-B)|_{\fr{n}_2} + u_3(-B)|_{\fr{n}_3},
\end{equation}
where $u_i, i = 0, 1, 2, 3$ belong to $\bb{R}^{+}$.  It is easy to see that the following relations hold: 
\begin{equation}\label{brackets2}
[\fr{n}_2, \fr{n}_2]\subset\fr{h}\oplus\fr{n}_3, \quad [\fr{n}_3, \fr{n}_3]\subset\fr{h}, \quad [\fr{n}_2, \fr{n}_3]\subset\fr{n}_2.
\end{equation}

\subsection{The Ricci tensor for metrics corresponding to inner products (\ref{metric1}) and (\ref{metric2})} 
From Lemma \ref{brackets} we see that the only non zero triples (up to permutation of indices) for the metric corresponding to (\ref{metric1}) are
\begin{equation}\label{triplets}
A_{111}, \quad A_{222}, \quad A_{1(12)(12)}, \quad A_{1(13)(13)}, \quad A_{2(12)(12)}, \quad A_{2(23)(23)}, \quad A_{(12)(23)(13)}
\end{equation}

We recall the following result  by A. Arvanitoyeorgos, V.V. Dzhepko and Yu.G. Nikonorov: 
\begin{lemma}\label{brac1} {\em (\cite{ADN1}, \cite{ADN2})} 
For $a, b, c = 1, 2, 3$ and $(a - b)(b - c) (c - a) \neq 0$ the following relations hold$:$
\begin{equation*}\label{eq14}
\begin{array}{lll} 
\displaystyle{A_{aaa} = \frac{k_a (k_a +1)(2k_a +1)}{n +1}},   &  \displaystyle{A_{(ab)(ab)a} = \frac{k_a  k_b (2k_a +1)}{(n +1)}}, 
 & \displaystyle{A_{(ab)(bc)(ac)} = \frac{2k_a  k_b  k_c}{n +1}}.  
\end{array} 
\end{equation*}
\end{lemma}

Now for the metric corresponding to (\ref{metric2}) we see that from  relations (\ref{brackets2}) and  [\cite{ArMoSa}, Proposition 6, p.\, 269]  the only non zero triples are: 
$$
A_{220}, \quad A_{330}, \quad A_{111}, \quad A_{122}, \quad A_{133}, \quad A_{322}.
$$
From A. Arvanitoyeorgos, K. Mori and Y. Sakane we have the following:
\begin{lemma}\label{brac2}{\em (\cite{ArMoSa})} The triples $A_{ijk}$ are given as follows:
\begin{equation*}\label{eq15}
\begin{array}{lll}
\displaystyle{A_{220} = \frac{d_2}{(d_2 + 4d_3)}}, & \displaystyle{A_{330} = \frac{4d_3}{(d_2 + 4d_3)}} &
\displaystyle{A_{111} = \frac{2d_3(2d_1 + 2 - d_3)}{(d_2 +4d_3)}}  \vspace{3pt}\\ 
\displaystyle{A_{122} = \frac{d_1d_2}{(d_2 + 4d_3)}} &
\displaystyle{A_{133} = \frac{2d_3(d_3 - 2)}{(d_2 + 2d_3)}} & \displaystyle{A_{322} = \frac{d_2d_3}{(d_2 + 4d_3)}}.
\end{array} 
\end{equation*}
\end{lemma}

By using the above lemmas, we obtain the components of the Ricci tensor for the metrics (\ref{metric1}) and (\ref{metric2}). 
\begin{prop}\label{prop1}
The components  of  the Ricci tensor ${r}$ for the invariant metric $\langle\cdot ,\cdot\rangle $ on Stiefel manifold $G/H$ defined by  $(\ref{metric1})$ are given as follows$:$  
\begin{equation}\label{ricci1}
\begin{array}{rcl} 
r_1 &=&  \displaystyle{\frac{k_1+1}{4 (n +1)  x_1} +
\frac{k_2}{4 (n +1) }  \frac{x_1}{{x_{12}}^2}} + \frac{k_3}{4 (n +1) } \frac{x_1}{{x_{13}}^2}, 
 \\   \\
 r_2 &=&   \displaystyle{\frac{k_2+1}{4 (n +1)  x_2} +
\frac{k_1}{4 (n +1) }  \frac{x_2}{{x_{12}}^2}} + \frac{k_3}{4 (n +1) } \frac{x_2}{{x_{23}}^2}, 
\\ \\
r_{12} &=&   \displaystyle{\frac{1}{ 2 x_{12}} +\frac{k_3}{4 (n +1)}\biggl(\frac{x_{12}}{x_{13} x_{23}} - \frac{x_{13}}{x_{12} x_{23}}  - \frac{x_{23}}{x_{12} x_{13}}\biggr) } -
\displaystyle{\frac{2 k_1 +1}{8 (n +1)}  \frac{ x_1}{{x_{12}}^2}-\frac{2 k_2 +1}{8 (n +1)}  \frac{ x_2}{{x_{12}}^2} },
\\ \\
r_{13} &=&    \displaystyle{\frac{1}{  2 x_{13}} +\frac{k_2}{4 (n +1)}\biggl(\frac{x_{13}}{x_{12} x_{23}} - \frac{x_{12}}{x_{13} x_{23}} - \frac{x_{23}}{x_{12} x_{13}}\biggr)}  -
\displaystyle{\frac{2 k_1 +1}{8 (n +1)}  \frac{ x_1}{{x_{13}}^2} },
\\  \\
r_{23}  &=&    \displaystyle{\frac{1}{ 2 x_{23}} +\frac{k_1}{4 (n+1)}\biggl(\frac{x_{23}}{x_{13} x_{12}} - \frac{x_{13}}{x_{12} x_{23}} - \frac{x_{12}}{x_{23} x_{13}}\biggr)} -
\displaystyle{\frac{2 k_2 +1}{8 (n +1)}  \frac{ x_2}{{x_{23}}^2}. } 
\end{array} 
\end{equation}
\end{prop}
To find Einstein metrics of the form (\ref{metric1}) reduces to find positive solutions of the system
\begin{equation}\label{equa1}
r_1 = r_2, \ r_2 = r_{12}, \ r_{12} = r_{13}, \ r_{13} = r_{23}.
\end{equation}

\begin{prop}\label{prop2}
The components  of  the Ricci tensor ${r}$ for the invariant metric $\langle\cdot ,\cdot\rangle $ on Stiefel manifold $G/H$ defined by  $(\ref{metric2})$ are given as follows$:$  
\begin{equation}\label{ricci2}
\begin{array}{lll} 
r_{0} &=& \displaystyle{\frac{u_0}{4u_{2}^{2}}\,\frac{d_{2}}{(d_2 + 4d_3)} + \frac{u_0}{4u_{3}^{2}}\,\frac{4d_3}{(d_2 + 4d_3)}}\\ \\
r_{1} &=& \displaystyle{\frac{1}{4d_{1}u_{1}}\,\frac{2d_3(2d_1 + 2 -d_3)}{(d_2 + 4d_3)} + \frac{u_1}{4 u_{2}^{2}}\,\frac{d_2}{(d_2 + 4d_3)} + \frac{u_1}{2d_1 u_{3}^{2}}\,\frac{d_3(d_3 - 2)}{(d_2 + 4d_3)}}\\ \\
r_{2} &=& \displaystyle{\frac{1}{2u_2} -\frac{u_3}{2 u_{2}^{2}}\,\frac{ d_3}{(d_2 + 4d_3)} -\frac{1}{2u_{2}^{2}}\left(u_{0}\,\frac{1}{(d_2 + 4d_3)} + u_{1}\,\frac{d_{1}}{(d_2 + 4d_3)}\right)} \\ \\
r_{3} &=& \displaystyle{\frac{1}{u_3}\,\left(\frac{1}{2} - \frac{1}{2}\,\frac{d_2}{(d_2 + 4d_3)}\right) + \frac{u_3}{4u_{2}^{2}}\,\frac{d_2}{(d_2 + 4d_3)} - \frac{1}{ u_{3}^{2}}\left(u_{0}\,\frac{2}{(d_2 + 4d_3)} + u_{1}\,\frac{d_3 - 2}{(d_2 + 4d_3)}\right)}\\ \\
\end{array}  
\end{equation}
\end{prop}
The metric of the form (\ref{metric2}) is Einstein if and only if the system:
\begin{equation}\label{equa2}
r_0 = r_1, \ r_1 = r_2, \ r_2 = r_3,
\end{equation}
has positive solutions.

\section{Einstein metrics on $V_{p}\bb{H}^{n}$}

In this section we solve the systems (\ref{equa1}) and (\ref{equa2}) for the various values of $k_i$, $i = 1,2,3$.  

\subsection{Einstein metrics for the inner products (\ref{metric1})}

\begin{theorem}
For $n=3, 4$ the Stiefel manifold $V_{2}\bb{H}^{n}$ admits:
\begin{itemize}
\item[(1)] Eight invariant Einstein metrics which are determined by $\Ad(\Sp(1)\times\Sp(1)\times\Sp(1))$-invariant inner products of the form (\ref{metric1}).  Four of them are new, two are Jensen's metrics and the other two are ADN metrics.
\item[(2)] Eight invariant Einstein metrics which are determined by $\Ad(\Sp(1)\times\Sp(1)\times\Sp(2))$-invariant inner products of the form (\ref{metric1}).  Four of them are new, two are Jensen's metrics and the other two are ADN metrics.
\end{itemize}
\end{theorem}
\begin{proof}
For $(1)$ we see that from Proposition \ref{prop1} the Ricci components of the metric corresponding to $\Ad(\Sp(1)\times\Sp(1)\times\Sp(1))$-invariant inner products of the form (\ref{metric1}) are given as follows:
\begin{eqnarray*}
&& r_1 = \frac{{x_1}}{16 {x_{12}}^2}+\frac{{x_1}}{16 {x_{13}}^2}+\frac{1}{8{x_1}}, \quad 
    r_2 = \frac{{x_2}}{16 {x_{12}}^2}+\frac{{x_2}}{16 {x_{23}}^2}+\frac{1}{8{x_2}}
   \\ \\
&& r_{12} = -\frac{3 {x_1}}{32 {x_{12}}^2}-\frac{3 {x_2}}{32 {x_{12}}^2}+\frac{1}{16}\left(\frac{{x_{12}}}{{x_{13}} {x_{23}}}-\frac{{x_{23}}}{{x_{12}}
   {x_{13}}}-\frac{{x_{13}}}{{x_{12}} {x_{23}}}\right)+\frac{1}{2 {x_{12}}},
   \end{eqnarray*}
   \begin{eqnarray*}
&& r_{13} = -\frac{3 {x_1}}{32 {x_{13}}^2}+\frac{1}{16}\left(-\frac{{x_{12}}}{{x_{13}}
   {x_{23}}}-\frac{{x_{23}}}{{x_{12}} {x_{13}}}+\frac{{x_{13}}}{{x_{12}} {x_{23}}}\right)+\frac{1}{2 {x_{13}}},\\ \\
&& r_{23} = \frac{1}{16}\left(-\frac{{x_{12}}}{{x_{13}} {x_{23}}}+\frac{{x_{23}}}{{x_{12}}
   {x_{13}}}-\frac{{x_{13}}}{{x_{12}} {x_{23}}}\right)-\frac{3 {x_2}}{32 {x_{23}}^2}+\frac{1}{2 {x_{23}}}.
\end{eqnarray*}
We consider the system of equations (\ref{equa1}) for $n=3$.  
Then the metric corresponding to $\Ad(\Sp(1)\times\Sp(1)\times\Sp(1))$-invariant inner products of the form (\ref{metric1}) is Einstein if the system (\ref{equa1}) has positive solutions.  We normalize our equations by putting $x_{23} = 1$.  Then we obtain the system of equations:
\begin{equation}\label{eksi1}
\begin{array}{lll}
f_1 &=& {x_1}^2 {x_{12}}^2 {x_2}+{x_1}^2 {x_{13}}^2 {x_2}-{x_1} {x_{12}}^2 {x_{13}}^2{x_2}^2-2 {x_1} {x_{12}}^2 {x_{13}}^2-{x_1} {x_{13}}^2 {x_2}^2+2 {x_{12}}^2{x_{13}}^2 {x_2} = 0\\
f_2 &=& 3 {x_1} {x_{13}} {x_2}-2 {x_{12}}^3 {x_2}+2 {x_{12}}^2 {x_{13}} {x_2}^2+4{x_{12}}^2 {x_{13}} +2 {x_{12}} {x_{13}}^2 {x_2}-16 {x_{12}} {x_{13}} {x_2}\\
&&+2{x_{12}} {x_2}+5 {x_{13}} {x_2}^2 = 0\\
f_3 &=& 3 {x_1} {x_{12}}^2-3 {x_1} {x_{13}}^2+4 {x_{12}}^3 {x_{13}}-16 {x_{12}}^2 {x_{13}}-4
   {x_{12}} {x_{13}}^3+16 {x_{12}} {x_{13}}^2-3 {x_{13}}^2 {x_2} = 0\\
f_4 &=& -3 {x_1} {x_{12}}+3 {x_{12}} {x_{13}}^2 {x_2}-16 {x_{12}} {x_{13}}^2+16 {x_{12}}
   {x_{13}}+4 {x_{13}}^3-4 {x_{13}} = 0
\end{array}
\end{equation}
We consider a polynomial ring $R = \bb{Q}[x_1, x_2, x_{12}, x_{13}]$ and an ideal $I$ generated by $\{f_1, f_2, f_3, f_4,$ $z\,x_1\,x_2\,x_{12}\,x_{13}-1\}$ to find non zero solutions of equations (\ref{eksi1}).  We take a lexicographic
order $>$ with $z > x_2 > x_1 > x_{12} > x_{13}$ for a monomial ordering on $R$.  Then, by the aid of computer, we see that a Gr\"obner basis for the ideal $I$ contains the polynomial
$$
(x_{13} - 1)h(x_{13}),
$$
where $h(x_{13})$ is a polynomial of $x_{13}$ given by
{\small\begin{eqnarray*}
&& h(x_{13}) = 26264641347161101886463 {x_{13}}^{30}-508287291683094283147326
   {x_{13}}^{29}\\
&&+4744919846389271826285855 {{x_{13}}}^{28} -28464304403498295853317720 {x_{13}}^{27}\\
&&+123660202199445490641611164 {x_{13}}^{26}-415707976104450072060636864 {x_{13}}^{25}\\
&&+1125839862148616037654823494 {x_{13}}^{24}-2515459993664189079183771508 {x_{13}}^{23}\\ &&+4689788438841035164098977324 {x_{13}}^{22}-7301053768516762755075524460 {x_{13}}^{21}\\
&& +9384221512521610297473108154 {x_{13}}^{20} -9649002945579106949449677656 {x_{13}}^{19}\\&&+7301130495287173304405589627 {x_{13}}^{18}-2928458036429380583757463446 {x_{13}}^{17}\\
&&-1384378004627046466599664225 {x_{13}}^{16} +3187252032549620236743974472 {x_{13}}^{15}\\
&&-1384378004627046466599664225 {x_{13}}^{14}-2928458036429380583757463446 {x_{13}}^{13}\\
&&+7301130495287173304405589627 {x_{13}}^{12}-9649002945579106949449677656 {x_{13}}^{11}\\
&&+9384221512521610297473108154 {x_{13}}^{10}-7301053768516762755075524460 {x_{13}}^9\\
&&+4689788438841035164098977324 {x_{13}}^8-2515459993664189079183771508 {x_{13}}^7\\
&&+1125839862148616037654823494 {x_{13}}^6-415707976104450072060636864 {x_{13}}^5\\
&&+123660202199445490641611164 {x_{13}}^4-28464304403498295853317720 {x_{13}}^3\\
&&+4744919846389271826285855 {x_{13}}^2-508287291683094283147326 x_{13}\\
&&+26264641347161101886463.
\end{eqnarray*} }
By solving the equation $h(x_{13}) = 0$ numerically, we obtain four positive solutions which are given approximately as $0.568723, \, 0.595776, \, 1.67848, \, 1.75833$.  We also see that the Gr\"obner basis for the ideal $I$ contains the polynomials
$$
x_{12} - \al_{12}(x_{13}), \quad x_{1} - \al_{1}(x_{12}, x_{13}), \quad x_{2} - \al_{2}(x_{12}, x_{13})
$$
where $\al_{12},\, \al_{1},\, \al_{2}$ are polynomials of $x_{13}$ and $x_{12}$ with rational coefficients. By substituting the solution of $x_{13}$ into $\al_{12},\, \al_{1}$ and $\al_{2}$ we obtain four positive solutions of the system of equations $\{f_1=0,\, f_2 = 0,\, f_3 = 0, \, f_4 = 0\}$ approximately as $(x_1,\, x_2,\, x_{12},\, x_{13},\, x_{23}) \approx$ 
\begin{eqnarray*} & & 
(0.276281,\, 0.251266,\, 0.460887,\, 0.568722,\, 1),\quad
 (1.112249,\, 0.417937,\, 1.598741,\, 0.595776,\, 1),\\ & & 
 (0.701500,\, 1.866891,\, 2.683459,\, 1.678482,\, 1),\quad 
 (0.441809,\, 0.485793,\, 0.810389,\, 1.758325,\, 1). \end{eqnarray*}
Now we consider the case where $x_{13} = 1$.  By substituting $x_{13} = 1$ into (\ref{eksi1}) and solving again numerically, we obtain solutions approximately as $(x_1,\, x_2,\, x_{12},\, x_{13},\, x_{23}) \approx$ 
\begin{eqnarray}\label{sol1}
& & (0.472797,\, 0.472797,\, 0.472797,\, 1,\, 1),\quad 
(1.812916,\, 1.812916,\, 1.812916,\, 1,\, 1), \nonumber\\
& & (0.344889,\, 0.344889,\, 0.80019,\, 1,\, 1),\quad  (0.483972,\, 0.483972,\, 2.585187,\, 1,\, 1).
\end{eqnarray}
The first two of the solutions (\ref{sol1}) are Jensen's metrics and the other two are ADN metrics.

To prove part $(2)$ of the theorem, we can work analogously.   In this case the system (\ref{equa1}) for $k_1 = k_2 = 1$, $k_3 = 2$ and $x_{23} = 1$ is the following:
\begin{equation}\label{case2}
\begin{array}{lll}
g_1 = 2 {x_1}^2 {x_{12}}^2 {x_2}+{x_1}^2 {x_{13}}^2 {x_2}-2 {x_1} {x_{12}}^2 {x_{13}}^2 {x_2}^2-2 {x_1} {x_{12}}^2 {x_{13}}^2-{x_1} {x_{13}}^2 {x_2}^2+2{x_{12}}^2 {x_{13}}^2 {x_2} = 0\\
g_2 = 3 {x_1} {x_{13}} {x_2}-4 {x_{12}}^3 {x_2}+4 {x_{12}}^2 {x_{13}} {x_2}^2+4{x_{12}}^2 {x_{13}}+4 {x_{12}} {x_{13}}^2 {x_2}-20 {x_{12}} {x_{13}} {x_2}\\
\quad \quad +4{x_{12}} {x_2}+5 {x_{13}} {x_2}^2 = 0\\
g_3 = 3{x_1} {x_{12}}^2-3 {x_1} {x_{13}}^2+6 {x_{12}}^3 {x_{13}}-20 {x_{12}}^2 {x_{13}}-6{x_{12}} {x_{13}}^3+20 {x_{12}} {x_{13}}^2-2 {x_{12}} {x_{13}}-3 {x_{13}}^2 {x_2} = 0\\
g_4 = -3 {x_1} {x_{12}}+3{x_{12}} {x_{13}}^2 {x_2}-20{x_{12}} {x_{13}}^2+20{x_{12}}{x_{13}}+4 {x_{13}}^3-4 {x_{13}} = 0.
\end{array}
\end{equation}
A Gr\"obner basis for the ideal $I$ of the polynomial ring $R = \bb{Q}[x_1, x_2, x_{12}, x_{13}]$ generated by $\{g_1, g_2, g_3, g_4, z\,x_1\,x_2\,x_{12}\,x_{13}-1 \}$ and equipped with the lexicographic order $>$ with $z > x_2 > x_1 > x_{12} > x_{13}$ for a monomial ordering on $R$, contains the polynomial $(x_{13} - 1)h(x_{13})$.  The degree of $h(x_{13})$ is $30$.  By performing analogous computations as in the proof of part (1) we obtain the following Einstein metrics which correspond to the $\Ad(\Sp(1)\times\Sp(1)\times\Sp(2))$-invariant inner products of the form (\ref{metric1}):  The new metrics $(x_1, x_{2}, x_{12}, x_{13}, x_{23}) \approx $
\begin{eqnarray*}
 & & (0.227002,\, 0.207491,\, 0.362198,\, 0.643984,\, 1),  \quad
  (1.293692,\, 0.292641,\, 1.707728,\, 0.683996,\, 1)\\
 & & (0.427841,\, 1.891372,\, 2.496690,\, 1.461995,\, 1), \quad 
 (0.322198,\, 0.352496,\, 0.562433,\, 1.552832,\, 1). 
\end{eqnarray*}
the Jensen's metrics: $(x_1, x_{2}, x_{12}, x_{13}, x_{23}) \approx $ 
\begin{eqnarray*}
 & & (0.357518,\, 0.357518,\, 0.357518,\, 1,\, 1), \quad 
 (1.864703,\, 1.864703,\, 1.864703,\, 1,\, 1)
\end{eqnarray*}
and the ADN metrics: $(x_1, x_{2}, x_{12}, x_{13}, x_{23}) \approx $ 
\begin{eqnarray*}
&  & (0.256403,\, 0.256403,\, 0.607404,\, 1,\, 1), \quad
 (0.309365,\, 0.309365,\, 2.398604,\, 1,\, 1). 
\end{eqnarray*}
\end{proof}
By working in a similar manner as in the above proof, we conjecture the existence of 
 new Einstein metrics on 
$V_{n-1}\bb{H}^{n}$ ($n\ge 3$) and $V_{n-2}\bb{H}^{n}$ ($n\ge 5$) 
 as shown in Table 1.
{\small
\begin{table}[htb]
\begin{tabular}{c|c|c}
\thickhline
 $V_{k_1+k_2}\bb{H}^{k_1+k_2+k_3}$    & {\bf Jensen's metrics} & {\bf New metrics}   \\
 \thickhline
 $k_1 = n-2, k_2 = k_3 = 1$ &   &\\
 \hline
$3 \leq n \leq 7$ & $2$ & $6$\\
$8 \leq n \leq 29$ & $2$ & $8$\\
$ 30 \leq n $  & $2$  & $10$\\
\thickhline
 $k_1 = n-3, k_2 =1,  k_3 = 2$ & &\\
 \hline
 $5 \leq n \leq 9$  & $2$  & $6$  \\
 $n = 10$  & $2$  & $8$  \\
 $11 \leq n \leq 27 $ & $2$  & $6$  \\
 $28 \leq n \leq 40$  & $2$  & $8$  \\
 $ 41 \leq n$   & $2$   & $10$\\
\thickhline             
 \end{tabular}
 \smallskip
 \caption{ {\small Conjectured number of Einstein metrics corresponding to $\Ad(\Sp(k_1)\times\Sp(k_2)\times\Sp(k_3))$-invariant inner products of the form (\ref{metric1})}}
\end{table}
}


\subsection{Einstein metrics for inner products (\ref{metric2})}
For the invariant metrics on $ \Sp(n)/\Sp(n-p)$ determined by the $\Ad(\U(p)\times \Sp(n-p))$-invariant scalar products (\ref{metric2}), if either $u_0 = u_1$, $u_0 = u_3$ or $u_1 = u_3$,  we  see that invariant Einstein metrics on $ \Sp(n)/\Sp(n-p)$ of the form (\ref{metric2}) are Jensen's metrics. 
\begin{theorem}
For $ \displaystyle 2  \leq p \leq \frac{3}{4} n$, there exist two  invariant Einstein metrics on $ \Sp(n)/\Sp(n-p)$  of the form {\em (\ref{metric2})} which are different from Jensen's metrics. 
\end{theorem}
\begin{proof}
From Proposition \ref{prop2} we see that the components of the Ricci tensor for the metric of the form (\ref{metric2}) are given by
\begin{eqnarray*}
&&r_{0} = \displaystyle{\frac{(n-p)}{4(n+1)}\frac{u_{0}}{{u_{2}}^{2}}  +\frac{p+1}{4(n+1)} \frac{u_{0}}{{u_{3}}^{2}}}, \quad \quad
r_{1} = \displaystyle{\frac{p}{8 (n+1)}\frac{1}{u_{1}} + \frac{(n-p)}{(n+1)}\frac{u_{1}}{4{u_{2}}^{2}}+ \frac{(p+2)}{  8(n+1)}\frac{ u_{1}}{{u_{3}}^{2}} },\\ \\ 
&&r_{2} = \displaystyle{\frac{1}{2u_{2}} - \frac{p+1}{8(n+1)} \frac{u_3}{u_{2}^{2}}\, - \frac{1}{8 p (n+1) {u_{2}}^{2}}\left(u_{0}  + (p^2-1) u_{1} \right)},\\ \\
&&r_{3} = \displaystyle{\frac{1}{u_{3}}\, \frac{p+1}{2(n+1)}+\frac{(n-p)}{4(n+1)} \frac{u_{3}}{u_{2}^{2}}\, - \frac{1}{2 p (n+1){u_{3}}^{2}}\,\left(u_{0} + u_{1}\,\frac{(p-1)(p+2)}{2}\right).}\\
\end{eqnarray*}
We put $u_{3} = 1$ so the system $\{r_0 = r_1, \ r_1 = r_2, \ r_2 = r_3\}$ is given by
\begin{equation}\label{syst2}
\begin{array}{lll}
f_{1}&=&  \ \ 2 {u_0} {u_1} (n-p)-2 {u_1}^2 (n-p)+2 (p+1) {u_0} {u_1}{u_2}^2-(p+2) {u_1}^2 {u_2}^2-p {u_2}^2 = 0\\
f_{2}&=& \ \ {u_1}^2 \left((2 n -p) p-1\right)-4 (n+1) p  {u_1} {u_2}+p^2 {u_2}^2 +p (p+2) {u_1}^2
   {u_2}^2 \\ & & + p (p+1) {u_1}+{u_0}{u_1} = 0\\
f_{3}&=& \ \ 4 (n+1) p {u_2}+p (-2 n+p-1)+ 2 (p-1)(p+2){u_1}{u_2}^2 -(p-1)(p+1){u_1}\\&  & - 4 p(p+1){u_2}^2+4{u_0}{u_2}^2-{u_0}= 0.
\end{array}  
\end{equation}

We consider a polynomial ring $R = \bb{Q}[z, u_0, u_1, u_2]$ and an ideal $I$ generated by $\{f_1,\, f_2,\, f_3,\, z\,u_{0}\,u_{1}\,u_{2} - 1\}$ to find non zero solutions of the above equations. We take a lexicographic order $>$ with $z > u_{0} > u_{2} > u_{1}$ for a monomial ordering on $R$. Then, by the aid of computer, we see that a Gr\"obner basis for the ideal $I$ contains the polynomial $(u_{1} -1)U_{1}(u_{1}),$ where $U_{1}$ is a polynomial of $u_{1}$ given by:
{\small \begin{eqnarray*}
&&U_{1}(u_1) = (4 n-p+1)^4 (p-1)^2 (p+2)^2 {u_1}^8\\ & & -2 (4 n-p+1)^3 (p-1) (p+2) \left(-10 p^3+2 n p^2-25 p^2-20 p+8 n\right) {u_1}^7\\ & &+(4 n-p+1)^2
   (140 p^6-20 n p^5+726 p^5-396 n^2 p^4-828
   n p^4+1057 p^4+352 n^3 p^3-272 n^2 p^3\\ & &-1860 n
   p^3+178 p^3+896 n^3 p^2+1248 n^2 p^2-840 n
   p^2-596 p^2+640 n^3 p+832 n^2 p-560 n p\\ & &-816
   p+512 n^3+1088 n^2+608 n-64)
   {u_1}^6 \\ & &-2 (4 n-p+1)(-182 p^7-166 n
   p^6-1407 p^6+2020 n^2 p^5+3174 n p^5-2570
   p^5-2272 n^3 p^4\\ & &+2412 n^2 p^4+10762 n p^4-504
   p^4+576 n^4 p^3-5680 n^3 p^3-5760 n^2 p^3+9190
   n p^3+1377 p^3\\ & &+1024 n^4 p^2-4928 n^3 p^2-11072
   n^2 p^2-500 n p^2-202 p^2+512 n^4 p-384 n^3
   p-2752 n^2 p\\ & &-688 n p-480 p-512 n^4-896 n^3-448
   n^2+128 n-32) {u_1}^5\\ & &+(46 p^8+2524 n p^7+2706 p^7-10740 n^2 p^6-5140 n
   p^6+7173 p^6+6624 n^3 p^5-41736 n^2 p^5\\ & &-63232 n
   p^5-5706 p^5+31936 n^4 p^4+160544 n^3
   p^4+138044 n^2 p^4-30916 n p^4-15743 p^4\\ & &-55296
   n^5 p^3-147968 n^4 p^3+60992 n^3 p^3+300208 n^2
   p^3+126100 n p^3+14418 p^3+24576 n^6 p^2\\ & &-27648 n^5 p^2-365312 n^4 p^2-527104 n^3 p^2-217536
   n^2 p^2-32888 n p^2-1358 p^2+57344 n^6 p\\ & &+202752
   n^5 p+196608 n^4 p+19072 n^3 p-30464 n^2
   p-12512 n p-1784 p+16384 n^6+65536 n^5 
     \end{eqnarray*} }
 {\small \begin{eqnarray*} 
   & &+90112
   n^4+52992 n^3+17344 n^2+3264 n+248)
   {u_1}^4 \\ 
 & &-2 (-330 p^8+274 n p^7-2129
   p^7+3408 n^2 p^6+8656 n p^6-2144 p^6-8960 n^3
   p^5-10088 n^2 p^5\\ & &+13352 n p^5+2146 p^5+11264
   n^4 p^4+10304 n^3 p^4-10032 n^2 p^4+8004 n
   p^4+5788 p^4
  \\ & &-9728 n^5 p^3-17536 n^4 p^3-20800
   n^3 p^3-42824 n^2 p^3-29634 n p^3-3701 p^3+4096
   n^6 p^2\\ & &+7168 n^5 p^2+17664 n^4 p^2+52928 n^3
   p^2+52752 n^2 p^2+13308 n p^2+1114 p^2+4096 n^6
   p\\ & &-4096 n^5 p-48128 n^4 p-68800 n^3 p -34208 n^2
   p-8856 n p-1000 p+8192 n^6+32768 n^5\\ & &+47104
   n^4+31616 n^3+12096 n^2+2576 n+256)
   {u_1}^3
   \\ & &+ (444 p^8-1564 n p^7+1226
   p^7+1452 n^2 p^6-5132 n p^6-121 p^6+1856 n^3
   p^5+11032 n^2 p^5\\ & &+1416 n p^5-1854 p^5-4672 n^4
   p^4-12192 n^3 p^4+1788 n^2 p^4+10132 n p^4+439
   p^4+2560 n^5 p^3\\ & &+2560 n^4 p^3-15584 n^3
   p^3-23664 n^2 p^3-3276 n p^3+646 p^3+3072 n^5
   p^2+14336 n^4 p^2\\ & &+14080 n^3 p^2-6048 n^2
   p^2-8056 n p^2-1020 p^2+5120 n^4 p+14336 n^3
   p+11072 n^2 p\\ & &+1200 n p-80 p+1792 n^3+3584
   n^2+1696 n+320 ) {u_1}^2\\ & &+2 (3
   p^2-4 n p-p-2) (18 p^6-70 n p^5+5
   p^5+124 n^2 p^4+34 n p^4-10 p^4-128 n^3 p^3 \\ & &-148
   n^2 p^3 -50 n p^3 -62 p^3+64 n^4 p^2+144 n^3 p^2+104 n^2 p^2
   +106 n p^2+9 p^2+32 n^3 p\\ & &+64 n^2 p+88 n p+56 p+16 ) {u_1}+ \left(4 n p-3 p^2+p+2\right)^2 \left(2 n p-p^2+p+1\right)^2.
\end{eqnarray*} }
If $u_1 \neq 1$ then $U_{1}(u_1) = 0$.  We will prove that the equation $U_{1}(u_1) = 0$ has at least two positive roots.  Observe that
$
U_{1}(0) = \left(4 n p-3 p^2+p+2\right)^2 \left(2 n p-p^2+p+1\right)^2
$
is positive for all $p \leq n$ and  
{\small \begin{eqnarray*}
& & U_{1}(1) = 64 \left(16 n^3 p-16 n^2 p^2+32 n^2 p+4 n^2-40 n
   p^2+4 n p+4 p^4+12 p^3-7 p^2\right)\times \\ 
   & &\left(16 n^3
   p+48 n^3-16 n^2 p^2-16 n^2 p+112 n^2-16 n p^2-56 n
   p+72 n+4 p^4+12 p^3+5 p^2-30 p+9\right) \\
   & & = \left( (32 p^2+32 p+4 ) (n-p)^2+ (16 p^3+24
   p^2+12 p) (n-p)+16 p (n-p)^3+4 p^4+4 p^3+p^2 \right)\times \\
    & & \left ((16 p+48) (n-p)^3+(32 p^2+128 p+112) (n-p)^2+ (16 p^3+96
   p^2+168 p+72) (n-p) \right.\\ & & \left.+4
   p^4+28 p^3+61 p^2+42 p+9\right)
\end{eqnarray*} }
is positive for all $p \leq n$. 
We also  see that 
$U_{1}(1/5) = \displaystyle -\frac{64}{390625} u(n, p), $
where 
{\small \begin{eqnarray*}
& &
u(n, p)=  160000 n^6 p^2-160000 n^6 p+640000
   n^6-403200 n^5 p^3+230400 n^5 p^2 \\
    & & -2376000 n^5
   p+2548800 n^5-214784 n^4 p^4+582912 n^4
   p^3+2333952 n^4 p^2-7877504 n^4 p \\
    & & +3695424
   n^4+1771264 n^3 p^5-927296 n^3 p^4+827840 n^3
   p^3+8098416 n^3 p^2-10284688 n^3 p \\
    & & +2114464
   n^3-2423104 n^2 p^6-388160 n^2 p^5-2433056 n^2
   p^4-1950896 n^2 p^3+9653884 n^2 p^2 \\
    & & -5443832 n^2
   p+135164 n^2+1436864 n p^7+1036416 n p^6+1028576 n
   p^5-1683360 n p^4 \\
    & & -3395380 n p^3+4423856 n
   p^2-652556 n p-194416 n-327184 p^8-375232
   p^7-23080 p^6 \\
    & & +859256 p^5+313895 p^4-1100766
   p^3+477967 p^2+177080 p-1936. 
\end{eqnarray*} }
We claim that, for $\displaystyle  p \leq \frac{3}{4} n$, $U_{1}(1/5)$
is negative.  Indeed, we expand  $u(n, p)$  at $\displaystyle n= \frac{4 p}{3}$ as follows:
{\small \begin{eqnarray*}
& &
u(n, p)=\left(160000 p^2-160000 p+640000\right)
   \left(n-\frac{4 p}{3}\right)^6\\
    & & +\left(876800
   p^3-1049600 p^2+2744000 p+2548800\right)
   \left(n-\frac{4 p}{3}\right)^5\\
    & & +\frac{64}{3}
   \left(63932 p^4-100676 p^3+166904 p^2+427242
   p+173223\right) \left(n-\frac{4
   p}{3}\right)^4
\end{eqnarray*}  }
{\small \begin{eqnarray*}   
    & & +\frac{16}{27} \left(1759952
   p^5-2206604 p^4+2322548 p^3+19232541 p^2+15903405
   p+3568158\right) \left(n-\frac{4
   p}{3}\right)^3\\
    & & +\frac{4}{27} (2691728
   p^6-23408 p^5-1386904 p^4+46116396 p^3 +53547669 p^2+20344662 p+912357 ) \left(n-\frac{4 p}{3}\right)^2\\
    & & +\frac{4}{81} (1202864
   p^7+3605920 p^6-5670856 p^5+38327736 p^4+51328395
   p^3\\
    & & +23978268 p^2-5915403 p-3936924 )
   \left(n-\frac{4 p}{3}\right)\\
    & & +\frac{1}{729}
   (383344 p^8+23887424 p^7-39989096
   p^6+136017720 p^5+182269791 p^4\\
    & & +96117138
   p^3-110673945 p^2-59881032 p-1411344 ). 
\end{eqnarray*} }
Then we see that the coefficients are polynomials of $p$ and are positive for $p \geq 1$.  
 Hence we see that  the equation $U_{1}(u_{1}) =0$  has at least two positive solutions $u_{1} = \al_{1}, \al_{2}$ with $0< \al_{1} < 1/5$ and  $1/5 < \al_{2} < 1$.  

Now, we consider a Gr\"obner basis (take a lexicographic order $>$ with $z > u_{0} > u_{2} > u_{1}$) for the ideal $J$ generated by the polynomials
$
\{ f_1, \, f_2, \, f_3, $  $  \, z\, u_{0}\, u_{1}\, u_{2} \,( u_{1} - 1)  -1\}.
$ 
This basis contains the polynomial $U_{1}(u_{1})$ and the polynomials 
$$
a_1(n, p)\, u_2 + W_{1}(u_1, n, p), \quad a_2(n, p)\, u_0 + W_{2}(u_1, n, p)
$$
where $a_i(n, p)$ ($i = 1, 2$) are polynomials of $n$ and $p$,  and $ W_{i}(u_1, n, p)$  $(i = 1, 2)$ are polynomials of $u_1$, $n$ and $p$.    For $1 \leq p < n$ we can see that the polynomials $a_{i}(n, p)$ ($i = 1,2$) are positive.  Thus, for the positive values $u_1 = \al_1, \al_2$  we obtain the real values $u_{2} = \gamma_1, \gamma_2$ and $u_0 = \beta_1, \beta_2$ as solutions of the system (\ref{syst2}).  We prove next that these solutions are positive.  We consider the ideal $J$ and we take the lexicographic order $>$ with $z > u_0 > u_1 > u_{2}$  for a monomial ordering on $R$.  Then  we see that the Gr\"obner basis  for the ideal $J$ contains the polynomial $U_{2}(u_{2})$:
{\small \begin{eqnarray*}
&&U_{2}(u_{2}) = 2 (p+1) (p+2) \left(5 p^3+8 p^2+p+2\right) {u_2}^8
-4 (n+1) (p+2) (3 p+1) \left(p^2+p+2\right) {u_2}^7 
\\ & &+  \left(2 n \left(17 p^4+58 p^3+61 p^2+20 p+4\right)-30
   p^5-101 p^4-88 p^3+11 p^2+28 p+4\right){u_2}^6 
   \\ & & -8 (n+1)(n-p) \left(4 p^3+9 p^2+9 p+6\right) {u_2}^5
   +2 (p+1) (n-p) \left(n (21 p+26) p-16 p^3-17 p^2+ \right. 
   \\ & & \left. 16 p+12\right) {u_2}^4
     -4 (n+1)  (7 p^2+8 p+4 ) (n-p)^2  
   (34 n p^4+116 n p^3+122 n p^2+40 n p +8 n-30 p^5 \\ & &  -101 p^4  
  -88 p^3+11 p^2+28 p+4 ) {u_2}^3  +(n-p)^2 \left(22 n p^2+24 n p-14 p^3-9 p^2+20 p+14\right){u_2}^2
    \\ & & -8 (n+1) p (n-p)^3  {u_2}+2 (n-p)^3 \left(2 n p-p^2+p+1\right). 
\end{eqnarray*} }
The above polynomial can be written as follows:
{\small \begin{eqnarray*}
&&U_{2}(u_{2}) = 2 (p + 1) (p + 2)
   \left (5 p^3 + 8 p^2 + p + 
    2 \right) {u_ 2}^8 \\& &- \left (\left (12 p^4 + 40 p^3 + 60 p^2 + 
       64 p + 16 \right) (n - p) + 12 p^5 + 52 p^4 + 100 p^3 + 124 p^2 + 80 p + 
    16 \right) {u_ 2}^7  \\& &+ \left (\left (34 p^4 + 116 p^3 + 122 p^2 + 
       40 p + 8 \right) (n - p) + 4 p^5 + 15 p^4 + 34 p^3 + 51 p^2 + 
    36 p + 4 \right) {u_ 2}^6  \\& &- \left (\left (32 p^3 + 72 p^2 + 
       72 p + 48 \right) (n - p)^2 + \left (32 p^4 + 104 p^3 + 
       144 p^2 + 120 p + 48 \right) (n - p) \right) {u_ 2}^5  \\& &
       + \left (\left (42 p^3 + 94 p^2 + 
       52 p \right) (n - p)^2 + \left (10 p^4 + 28 p^3 + 50 p^2 + 
       56 p + 24 \right) (n - p) \right) {u_ 2}^4  \\& &
       - \left (\left (28 p^2 + 32 p +  16 \right) (n - p)^3 + \left (28 p^3 + 60 p^2 + 48 p + 
       16 \right) (n - p)^2 \right) {u_ 2}^3  \\& &
       + \left (\left (22 p^2 + 24 p \right) (n - p)^3 
       + \left (8 p^3 + 15 p^2 + 20 p + 14 \right) (n - p)^2 \right) {u_ 2}^2 
        \\& &- \left (8\left (p^2 + p \right) (n - p)^3 + 8 p (n - p)^4 \right) {u_ 2}  + 4 p (n - p)^4 +
 2\left (p^2 + p + 1 \right) (n - p)^3.  
\end{eqnarray*} }

Then we see that the coefficients of the polynomial $U_{2}(u_{2})$ are positive for even degree and negative for odd degree terms.  Thus, if the equation $U_{2}(u_{2}) = 0$ has real solutions, then these are all positive.  So the solutions $u_{2} = \gamma_{1}, \gamma_{2}$ are positive.  Now if we take the lexicographic order $>$ with $z > u_{2} > u_{1} > u_{0}$ for a monomial ordering on $R$ we see that the Gr\"obner basis for ideal $J$ contains the polynomial $U_{0}(u_{0})$:
\begin{equation*}
\begin{array}{lll}
&& \displaystyle U_{0}(u_{0}) =\sum^{8}_{j= 0} b_j(n, p) {u_0^{}}^{j}_{}, \end{array}  
\end{equation*}
where $b_j(n, p) $ $(j = 0, \ldots,  8)$ are polynomials of $ n, p$ given by 

{\footnotesize \begin{eqnarray*}
& & b_8(n, p) = (4 n-p+1)^4 \left(4 n p-3 p^2+p+2\right), \\ 
&&-------------------------------------\\
& & b_7(n, p) = -4 (4 n-p+1)^3 \left(4 n p-3 p^2+p+2\right)\left(8 \left(p^3+p^2-p\right)
 \left(n-\frac{4 p}{3}\right)^2 \right.\\
 & & +\frac{2}{3} \left(20 p^4+35 p^3+p^2-6\right) \left(n-\frac{4 p}{3}\right)+\frac{1}{9}
 \left.  \left(32 p^5+80 p^4+64 p^3+81 p^2+15 p\right) \right),   \\ 
&&-------------------------------------\\
& & b_6(n, p) = 2 (4 n-p+1)^2 \Big( 256 p^2 \left(p^3+3 p^2+8 p+4\right) \left(n-\frac{4
   p}{3}\right)^5 +\frac{128}{3} p \left(31 p^5+99 p^4 \right) \\
 & &+251 p^3  +229 p^2+105 p\left. +24\right) \left(n-\frac{4
   p}{3}\right)^4  +\frac{16}{9} \left(1453 p^7\right. +5055 p^6+12251
   p^5+14491 p^4  \\ & & +11064 p^3+5682 p^2 \left. +1620 p+144\right)
   \left(n-\frac{4 p}{3}\right)^3  +\frac{8}{27} \left(7682 \right. 
   p^8+29130 p^7+69385 p^6 \\ & & +93512 p^5  +91377 p^4+66996
   p^3+34938 p^2+11556 p  \left. +1836\right) \left(n-\frac{4
   p}{3}\right)^2 +\frac{4}{81} \left(17264 p^9 \right.  \\ & & +70326
   p^8 +170383 p^7+251516 p^6+281415 p^5+238755 p^4+159678
   p^3+82647 p^2  \\ & &  +29700 p\left.  +6156\right) \left(n-\frac{4
   p}{3}\right)   +\frac{1}{243} \left(24220 p^{10}+\right. 101688
   p^9+246461 p^8+374056 p^7 \\ & &  +441102 p^6+375186 p^5+248805
   p^4+99990 p^3-864 p^2 \left.   -17172 p-7776\right) \Big), 
\\  
&&-------------------------------------\\
& & b_5(n, p) = -4 (1 + 4 n - p) \Big( 
512 \big((p-2)^7+18 (p-2)^6+145 (p-2)^5+665 (p-2)^4
 \\ & & +1841
   (p-2)^3+3036 (p-2)^2+2740 (p-2)+1040 \big)
   \left(n-\frac{4 p}{3}\right)^6+\big(2944 (p-2)^8  \\
 & & +59392
   (p-2)^7 +539648 (p-2)^6+2855808 (p-2)^5+9527680
   (p-2)^4 \\ & & +20345728 (p-2)^3+26992640 (p-2)^2+20261376
   (p-2)+6572032\big) \left(n-\frac{4 p}{3}\right)^5  \\ & &
   +\frac{32}{3} \big(656 (p-2)^9+14659
   (p-2)^8+148284 (p-2)^7 +886309 (p-2)^6+3429435
   (p-2)^5  \\ & &+8861566 (p-2)^4+15226429 (p-2)^3+16721530
   (p-2)^2+10623556 (p-2) \\ & & +2969008\big) \left(n-\frac{4
   p}{3}\right)^4+\frac{8}{27} \big(425380 (p-1)^9+2863646
   (p-1)^8+11577253 (p-1)^7 \\ & &+30611791 (p-1)^6+54558303
   (p-1)^5+65610147 (p-1)^4+51897508 (p-1)^3  \\ & &+25321372
   (p-1)^2+6590944 (p-1)+611632\big) \left(n-\frac{4
   p}{3}\right)^3+\frac{8}{27} \big(6379969
   (p-2)^9 \\ & & +46794397 (p-2)^8+228835788 (p-2)^7+781428193
   (p-2)^6+1896404686 (p-2)^5  \\ & & +3261741064 (p-2)^4+3884236767
   (p-2)^3+3038186126 (p-2)^2+1397340138
   (p-2) \\ & & +283999648\big) \left(n-\frac{4
   p}{3}\right)^2+\frac{2}{81} \big(239851565
   (p-2)^9+1292327534 (p-2)^8  \\
    & & +4925962096 (p-2)^7+13582984080
   (p-2)^6+27208043881 (p-2)^5+39118134870
   (p-2)^4  
   \end{eqnarray*} }
{\footnotesize \begin{eqnarray*}    
   & & +39117496470 (p-2)^3+25567903940
   (p-2)^2+9637883888 (p-2)+1529459552\big)
   \left(n-\frac{4 p}{3}\right) \\
    & &+\frac{1}{729}
   \big(3992109985 (p-2)^9+16666847996 (p-2)^8+50919921029
   (p-2)^7 \\ & & +114707778059 (p-2)^6+189167284379
   (p-2)^5+222947172468 (p-2)^4  \\ & &+178718186969
   (p-2)^3+88184923605 (p-2)^2+21024910874
   (p-2)+644929528 \big)
\Big), \\
&&-------------------------------------\\
 & & b_4(n, p) = \big(4096 p^8+24576 p^7+86016 p^6+131072 p^5+245760
   p^4+294912 p^3 +65536 p^2 \big) \\ & & \times \left(n-\frac{4
   p}{3}\right)^8+\frac{2048}{3} \big(49 p^9+321 p^8+1233
   p^7+2261 p^6+3750 p^5+4782 p^4+2740 p^3  +768 p^2\\ & & +96
   p \big) \left(n-\frac{4 p}{3}\right)^7+\frac{256}{9}
    \big(4105 p^{10}+28932 p^9+118842 p^8+247316 p^7+400641
   p^6    +523896 p^5 \\ & & +417748 p^4+207240 p^3+67440 p^2+11808
   p+576 \big) \left(n-\frac{4 p}{3}\right)^6+\frac{256}{27} \big(24098 p^{11}   +180444
   p^{10} \\ & &  +778551 p^9+1777939 p^8+2922618 p^7+3878928
   p^6+3651965 p^5+2429397 p^4   +1195056 p^3 \\ & & +403596 p^2+79488
   p+6912 \big) \left(n-\frac{4 p}{3}\right)^5+\frac{128}{81} \big(174458 p^{12}+1372854
   p^{11}  +6131907 p^{10} \\ & & +15018679 p^9+25376340 p^8+34044954
   p^7+35115293 p^6+27920751 p^5  +17703822 p^4  \\ & & +8548038
   p^3+2855736 p^2+589464 p+57024\big) \left(n-\frac{4
   p}{3}\right)^4+\frac{32}{243} \big(1605292
   p^{13}\\ & &  +13148268 p^{12}  +60007947 p^{11}+154833551
   p^{10}+269354535 p^9+363913566 p^8+393754726
   p^7  \\ & & +350780631 p^6+265670607 p^5+163840680 p^4+75798837
   p^3+24414912 p^2+4825656 p  \\ & & +402408 \big) \left(n-\frac{4
   p}{3}\right)^3+\frac{16}{729}  \big(4614376
   p^{14}+38994072 p^{13}+179809266 p^{12}+481672982
   p^{11} \\ & &  +860592570 p^{10}+1167204042 p^9+1287983179
   p^8+1232426736 p^7+1067917563 p^6  \\ & & +795423942 p^5+465504435
   p^4+200889342 p^3+59178519 p^2+10162260 p+790236 \big)
   \left(n-\frac{4 p}{3}\right)^2\\
\\ & & +\frac{8}{2187} \big(7627984
   p^{15}+65954616 p^{14}+304768332 p^{13}+838750310
   p^{12}+1537511100 p^{11}  \\ & & +2100643851 p^{10}+2341188529
   p^9+2357827446 p^8+2283500553 p^7+1980427914
   p^6 \\ & & +1382011065 p^5+732725784 p^4+280881675 p^3+72047799
   p^2+11519658 p+892296 \big) \left(n-\frac{4 p}{3}\right) \\ & & +\frac{1}{6561} \big(22391824 p^{16}+196359072
   p^{15}+904687368 p^{14}+2541976256 p^{13}+4803063945
   p^{12}\\ & & +6757501974 p^{11}+7870887961 p^{10}+8606912064
   p^9+9296805819 p^8+8981127162 p^7\\ & & +6981506001
   p^6+4203824292 p^5+1914330996 p^4+634122108 p^3+148364622
   p^2 \\ & &+22149936 p+1627128 \big),
\\ 
&&-------------------------------------\\
& & b_3(n, p) = -4 \Big(\left(6144 p^8+26624 p^7+92160 p^6+161792 p^5+106496
   p^4+16384 p^3+16384 p^2\right) \\ & &\times \left(n-\frac{4
   p}{3}\right)^8+\frac{512}{3} \big(273 p^9+1307 p^8+4563
   p^7+8813 p^6+8036 p^5+3304 p^4+1240 p^3+432 p^2  \\ & & +96 p \big ) \left(n-\frac{4 p}{3}\right)^7+\frac{128}{9}
   \big(10530 p^{10}+54451 p^9+192546 p^8+397024 p^7+434848
   p^6+259613 p^5 \\ & & +112820 p^4+44496 p^3+15864 p^2+4752
   p+288 \big) \left(n-\frac{4 p}{3}\right)^6+\frac{8}{729}
   \big(8070816 (p-1)^{14}  \\ & & +162347188 (p-1)^{13}+1543182304
   (p-1)^{12}+9135183025 (p-1)^{11}+37388517475
   (p-1)^{10}  \\ & & +111258986177 (p-1)^9+247002835859
   (p-1)^8+413979949333 (p-1)^7  \\ & & +524851826468
   (p-1)^6+499980342751 (p-1)^5+352066978464
   (p-1)^4 \\ & & +177974328764 (p-1)^3+61419614640
   (p-1)^2+13149220928 (p-1)+1359883264 \big)  \\ & & \times 
   \left(n-\frac{4 p}{3}\right)^2+\frac{2}{2187} \big(23906064
   (p-2)^{15}+866418488 (p-2)^{14}+14699562236
   (p-2)^{13}
\end{eqnarray*} }
{\footnotesize \begin{eqnarray*}     
   & &  +154716831646 (p-2)^{12}+1128622433268
   (p-2)^{11}   +6037636612741 (p-2)^{10}  \\ & & +24442007879239
   (p-2)^9+76163670514689 (p-2)^8+183984545006253
   (p-2)^7 \\ & &  +344157568970319 (p-2)^6+493885552028239
   (p-2)^5+533343522444589 (p-2)^4  \\ & &  +419010880295843
   (p-2)^3+225773544980448 (p-2)^2+74489301958682
   (p-2)  \\ & &  +11323979370056 \big) \left(n-\frac{4
   p}{3}\right)+\frac{128}{27} \big(56544
   p^{11}+310214 p^{10}+1106496 p^9+2378078 p^8  \\ & &  +2901275
   p^7+2104141 p^6+1107319 p^5+510813 p^4+218826 p^3+93474
   p^2+26676 p+3456 \big) \\ & & \times \left(n-\frac{4 p}{3}\right)^5+\frac{32}{81} \big(745440 p^{12}+4280770
   p^{11}+15245253 p^{10}+33332221 p^9+42879460 p^8  \\ & & +33980882  p^7+19792277 p^6+10467657 p^5+5460186 p^4+2954718
   p^3+1326888 p^2+385992 p \\ & & +59616 \big) \left(n-\frac{4
   p}{3}\right)^4+\frac{8}{243} \big(6221520
   p^{13}+37017932 p^{12}+129677616 p^{11}+279941813
   p^{10}  \\ & & +358391423 p^9+276796984 p^8+149657704 p^7+79836921
   p^6+54122511 p^5+42334110 p^4\\ & &    +27686430 p^3+12549384
   p^2+3621996 p+480168 \big) \left(n-\frac{4 p}{3}\right)^3+\frac{1}{6561}\big(15502032 (p-2)^{16} \\ & & 
   +594225632 (p-2)^{15}+10682883096 (p-2)^{14}+119460467888
   (p-2)^{13} \\ & &  +929124476957 (p-2)^{12}+5324019350794
   (p-2)^{11}+23224120188504 (p-2)^{10} \\ & &  +78573533111735
   (p-2)^9   +208092670261551 (p-2)^8+432175097670381
   (p-2)^7 
   \\ & & +700253199489961 (p-2)^6+873949665477915
   (p-2)^5    +821258270676976 (p-2)^4  \\ & & +559567017398773
   (p-2)^3+259270566907163 (p-2)^2+72376779210930
   (p-2) \\ & & +9043636403392\big)\Big), 
\\ 
&&-------------------------------------\\
& & b_2(n, p) = 2 \Big(\big(49152 p^7+163840 p^6+245760 p^5+262144 p^4+65536
   p^3 \big) \left(n-\frac{4 p}{3}\right)^9\\ & &  + \big(460800
   p^8+1851392 p^7+3360768 p^6+3985408 p^5+2187264
   p^4+376832 p^3 \big) \left(n-\frac{4 p}{3}\right)^8 \\ & &  
   +\frac{1024}{3} \big(5403 p^9+25063
   p^8+52965 p^7+70829 p^6+53652 p^5+19528 p^4+2496 p^3-48
   p^2 \big)\\ & & \times \left(n-\frac{4 p}{3}\right)^7+\frac{256}{9}
   \big(145599 p^{10}+755751 p^9+1803240 p^8+2693636
   p^7+2483131 p^6 \\ & & +1301357 p^5+340194 p^4+27276 p^3-3240
   p^2-288 p \big) \left(n-\frac{4 p}{3}\right)^6+\frac{128}{27} \big(1217196
   p^{11} \\ & & +6900068 p^{10}+18123783 p^9+29729521 p^8+31437354
   p^7+20579762 p^6+7722099 p^5 \\ & & +1317717 p^4 -47088 p^3-50652
   p^2-6912 p \big) \left(n-\frac{4 p}{3}\right)^5+\frac{64}{81} \big(6577044
   (p-1)^{12} \\ & & +118846646 (p-1)^{11}+986151751
   (p-1)^{10}+4971372165 (p-1)^9+16952334168
   (p-1)^8   \\ & &+41153296820 (p-1)^7+72822618167
   (p-1)^6+94485465333 (p-1)^5 +89054474614
   (p-1)^4  \\ & &+59359625980 (p-1)^3+26515906016
   (p-1)^2+7115536800 (p-1)+866081984 \big) \left(n-\frac{4
   p}{3}\right)^4  \\ & &+\frac{16}{243} \big(46102536
   (p-1)^{13}+893886380 (p-1)^{12}+8009668926
   (p-1)^{11}+43931708085 (p-1)^{10}   \\ & & +164515487580
   (p-1)^9+443853483382 (p-1)^8+886597009628
   (p-1)^7\\ & &   +1325763408157 (p-1)^6+1481799690902
   (p-1)^5+1220624912368 (p-1)^4\\ & &+719031079864
   (p-1)^3+286287517072 (p-1)^2+68910586496
   (p-1)+7554724096 \big)\\   
   & & \times  \left(n-\frac{4
   p}{3}\right)^3+\frac{16}{729} \big (50554800
   (p-1)^{14}+1042222652 (p-1)^{13}+9979234256
   (p-1)^{12} \\ & & +58831030897 (p-1)^{11}+238508655079
   (p-1)^{10}+702977316389 (p-1)^9  \\ & &+1551902086965
   (p-1)^8+2603556297720 (p-1)^7+3330426101016
   (p-1)^6 
\end{eqnarray*} }
{\footnotesize \begin{eqnarray*}    
   & &+3226707834546 (p-1)^5+2325808006688
   (p-1)^4+1206279110888 (p-1)^3\\ & &  +424160371456
   (p-1)^2+90114540928 (p-1)+8673425408 \big)
   \left(n-\frac{4 p}{3}\right)^2 \\ & &+\frac{4}{2187} \big(125304816
   (p-1)^{15}+2722292056 (p-1)^{14}+27568091156
   (p-1)^{13} \\
    & &+172620320070 (p-1)^{12}+747161024728
   (p-1)^{11}+2366345497299 (p-1)^{10}  \\ & &+5659184311954
   (p-1)^9+10391520423657 (p-1)^8 +14742382956732
   (p-1)^7\\ & & +16117344646500 (p-1)^6+13420607976808
   (p-1)^5+8317499536784 (p-1)^4 \\ & &+3686792531040
   (p-1)^3+1090498987456 (p-1)^2+188447185920
   (p-1) \\ & &+13718142976 \big) \left(n-\frac{4
   p}{3}\right)+\frac{1}{19683}(393910608
   (p-2)^{16}+15229116448 (p-2)^{15} \\ & &+275646193992
   (p-2)^{14}+3100284455968 (p-2)^{13}+24249634772905
   (p-2)^{12} 
\\ & &+139844686276206 (p-2)^{11}+614944316806352
   (p-2)^{10}+2102754931172164 (p-2)^9 \\ & &+5648681203773822
   (p-2)^8+11955552983145732 (p-2)^7+19859975096618684
   (p-2)^6\\ & &+25602952630231152 (p-2)^5+25089976494997385
   (p-2)^4+18047962971277118 (p-2)^3 \\ & &+8974256392885572
   (p-2)^2+2750739644643676 (p-2)+390518428891112)
\Big), 
\\ 
&&-------------------------------------\\
& & b_1(n, p) = -4 (1 + p)\Big(
\left(65536 p^6+163840 p^5+98304 p^4+65536 p^3\right)
   \left(n-\frac{4 p}{3}\right)^9\\ & & +\big(589824 p^7+1867776
   p^6+1867776 p^5+1179648 p^4+393216 p^3\big)
   \left(n-\frac{4 p}{3}\right)^8  \\ & &+\left(2286080 p^8+8690688
   p^7+11844096 p^6+9195520 p^5+4462592 p^4+966656
   p^3\right) \left(n-\frac{4 p}{3}\right)^7 \\ & &+\frac{128}{9}
   \big(352586 p^9+1551023 p^8+2630223 p^7+2519177
   p^6+1551591 p^5+567228 p^4+89244 p^3 \\
    & & +144 p^2 \big )
   \left(n-\frac{4 p}{3}\right)^6+\frac{256}{9} \big(241848
   p^{10}+1198141 p^9+2397385 p^8+2739873 p^7+2044315
   p^6 \\ & & +996450 p^5+282204 p^4+34920 p^3+288 p^2 \big)
   \left(n-\frac{4 p}{3}\right)^5+\frac{32}{27}
   \big(5183216 p^{11}+28318098 p^{10}  \\ & & +64527951
   p^9+85185889 p^8+74405822 p^7+44421800 p^6+17327475
   p^5+3959649 p^4  \\ & &+424512 p^3+8532 p^2 \big)
   \left(n-\frac{4 p}{3}\right)^4+\frac{8}{243}
  \big(108797264 p^{12}+644522804 p^{11}+1629982860
   p^{10} \\ & &+2420187287 p^9+2406568896 p^8+1679806635
   p^7+813030876 p^6+259492401 p^5+50300136 p^4  \\ & &+4946589
   p^3+85536 p^2-10692 p \big) \left(n-\frac{4
   p}{3}\right)^3+\frac{8}{243} \big(39979920
   p^{13}+253174084 p^{12} \\ & & +696564112 p^{11}+1138779993
   p^{10}  +1260935722 p^9+998072517 p^8+566765046
   p^7  \\ & &+225114543 p^6+59846670 p^5+9584811 p^4+536058
   p^3-98172 p^2-21384 p\big ) \left(n-\frac{4 p}{3}\right)^2 \\ & & +\frac{4}{729} \big(50280776
   p^{14}+335998500 p^{13}+989279874 p^{12}+1750818979
   p^{11}+2123240909 p^{10} \\ & & +1867929560 p^9+1203732324
   p^8+558263160 p^7+178897338 p^6+34688142 p^5+1198962
   p^4   \\ & & -1440747 p^3-434727 p^2-53946 p \big) \left(n-\frac{4
   p}{3}\right)+\frac{1}{19683}(489042736 (p-1)^{15}  \\ & &+10741721536
   (p-1)^{14}+109618245032 (p-1)^{13}+690084927664
   (p-1)^{12}  \\ & &+2999595396751 (p-1)^{11}+9542417798552
   (p-1)^{10}  +22963667848149 (p-1)^9   \\ & &+42583235164476
   (p-1)^8+61361832092016 (p-1)^7+68713003505656
   (p-1)^6 \\ & &  +59299589725888 (p-1)^5+38721416081312
   (p-1)^4+18510971432512 (p-1)^3   \\ & &+6112599731968
   (p-1)^2+1245725978624 (p-1)+117991145472)
\Big),  
\\ 
&&-------------------------------------\\
\end{eqnarray*} }
{\footnotesize \begin{eqnarray*}    
& & b_0(n, p) = (2 + p + 8 p^2 + 5 p^3) \Big(
\left(32768 p^4+65536 p^3\right) \left(n-\frac{4 p}{3}\right)^9+\big(294912 p^5+786432 p^4
   \\ & & +393216
   p^3 \big) \left(n-\frac{4 p}{3}\right)^8+\big(1146880
   p^6+3801088 p^5+3506176 p^4+983040 p^3 \big)
   \left(n-\frac{4 p}{3}\right)^7 
   \\ & &+\frac{256}{9} \big(89069
   p^7+350806 p^6+455973 p^5+244692 p^4+47088 p^3+216
   p^2 \big) \left(n-\frac{4 p}{3}\right)^6
\\ & &+\big(3513856
   p^8+\frac{143439872 p^7}{9}+\frac{235910656
   p^6}{9}+20261888 p^5+7532544 p^4+1134592 p^3 \\ & &+24576
   p^2 \big) \left(n-\frac{4 p}{3}\right)^5+\frac{256}{27}
   \big(335951 p^9+1714250 p^8  +3364529 p^7+3336384
   p^6+1797066 p^5 \\ & & +516528 p^4+71712 p^3+3888 p^2 \big)
   \left(n-\frac{4 p}{3}\right)^4+\frac{32}{243}
   \big(14384188 p^{10}+81095072 p^9+183025959 p^8 
   \\ & &
   +218550762 p^7+151598655 p^6  +62898390 p^5+15659973
   p^4+2346894 p^3+195129 p^2+1458 p \big) 
   \\ & &\times \left(n-\frac{4 p}{3}\right)^3+\frac{32}{243} \big(5424180
   p^{11}+33305912 p^{10}+84256741 p^9+116455788
   p^8+97611129 p^7 \\ & &+52111152 p^6+18296307 p^5+4361364
   p^4+706887 p^3+68040 p^2+2916 p \big) \left(n-\frac{4 p}{3}\right)^2
 \\ & &  +\frac{32}{729} \big(3532256
   p^{12}+23347544 p^{11}+64988012 p^{10}+101259894
   p^9+98684001 p^8+63802980 p^7 \\ & &+28499985 p^6+9061956
   p^5+2056995 p^4+325620 p^3+40095 p^2+4374 p \big)
   \left(n-\frac{4 p}{3}\right) \\ & &+\frac{1}{19683}(
289238768
   p^{13}+2039206624 p^{12}+6165039240 p^{11}+10647882432
   p^{10}+11790525807 p^9 \\ & &+8913187026 p^8+4784163588
   p^7+1862446284 p^6+528299010 p^5+109752408 p^4+15825132
   p^3 \\ & &+236196 p^2-373977 p+39366) \Big). 
\end{eqnarray*} }

Thus, for $2  \leq p \leq \frac{3}{4} n$, we see that the coefficients $b_j(n, p)$ ($j =0, \ldots, 8$)  of the polynomial $U_{0}(u_{0})$ are positive for even degree and negative for odd degree terms and hence,  if the equation $U_{0}(u_{0}) = 0$ has real solutions, then these are all positive.  In particular,  the solutions $u_{0} = \beta_{1}, \beta_{2}$ are positive.  

Hence we see that the solutions of the system (\ref{syst2}) are of the form:
$$
\{u_{0} = \beta_1,\, u_{1} = \al_1,\, u_{2} = \gamma_1,\, u_{3} = 1\} \ \mbox{and}\ \{u_{0} = \beta_2,\, u_1 = \al_2, \, u_{2} = \gamma_2, \, u_{3} = 1\}
$$
and satisfy $\al_1, \al_2 \neq 1$
\end{proof}

\end{document}